%% file: main.tex
\documentclass{article}

\usepackage{microtype}
\usepackage{graphicx}
\usepackage{subcaption}
\usepackage{booktabs} %
\usepackage{makecell}
\usepackage{hyperref}

\usepackage[preprint]{icml2026}
\input{header}

\usepackage{amsmath}
\usepackage{amssymb}
\usepackage{mathtools}
\usepackage{amsthm}

\theoremstyle{plain}
\newtheorem{theorem}{Theorem}[section]

\newtheorem{lemma}[theorem]{Lemma}

\theoremstyle{definition}
\newtheorem{definition}[theorem]{Definition}
\newtheorem{assumption}[theorem]{Assumption}
\theoremstyle{remark}
\newtheorem{remark}[theorem]{Remark}

\crefname{problem}{problem}{problems}
\Crefname{problem}{Problem}{Problems}
\creflabelformat{problem}{(#2#1#3)}
\crefname{lemma}{Lemma}{Lemmas}  
\Crefname{lemma}{Lemma}{Lemmas}

\crefname{app}{Appendix}{Appendices}
\Crefname{app}{Appendix}{Appendices}

\usepackage[textsize=tiny]{todonotes}

\icmltitlerunning{Decentralized Optimization with Mixed Affine Constraints}

\begin{document}

\twocolumn[
  \icmltitle{Complexity of Decentralized Optimization with Mixed Affine Constraints}

  \begin{icmlauthorlist}
    \icmlauthor{Demyan Yarmoshik}{miriai}
    \icmlauthor{Nhat Trung Nguyen}{miriai}
    \icmlauthor{Alexander Rogozin}{miriai}
    \icmlauthor{Alexander Gasnikov}{miriai,innop,hdi}
  \end{icmlauthorlist}

  \icmlaffiliation{miriai}{MIRIAI, Moscow, Russia}
  \icmlaffiliation{innop}{Innopolis University, Innopolis, Russia}
  \icmlaffiliation{hdi}{HDI Lab, HSE University, Moscow, Russia}

  \icmlcorrespondingauthor{Demyan Yarmoshik}{yarmoshik.d@miriai.org}

  \icmlkeywords{Machine Learning, ICML}

  \vskip 0.3in
]

\printAffiliationsAndNotice{}  %

\begin{abstract}
This paper considers decentralized optimization of convex functions with mixed affine equality constraints involving both local and global variables. Constraints on global variables may vary across different nodes in the network, while local variables are subject to coupled and node-specific constraints. Such problem formulations arise in machine learning applications, including federated learning and multi-task learning, as well as in resource allocation and distributed control. We analyze this problem under smooth and non-smooth assumptions, considering both strongly convex and general convex objective functions. Our main contribution is an optimal algorithm for the smooth, strongly convex regime, whose convergence rate matches established lower complexity bounds. We further provide near-optimal methods for the remaining cases. 
\end{abstract}

\section{Introduction}

We consider distributed optimization problems in which the objective function is decomposed across multiple computational nodes and the decision variables are subject to both local and global affine constraints. Specifically, we consider problems of the form
\begin{align}\label{prob:opt_mixed_constraints_intro}
	\min_{x_1, \ldots, x_n, \tilde x}~ &\sumin f_i(x_i, \tilde x) \tag{P} \\
	\text{s.t.}~ &\sumin (A_i x_i - b_i) = 0,~ C_i x_i = c_i,~ \widetilde C_i \tilde x = \tilde c_i. \nonumber
\end{align}
There are two groups of variables: $x_i \in \R^{d_i}$ are individual for each node while $\tilde x \in \R^{\tilde{d}}$ is common for all the nodes.
Each node locally owns $x_i, A_i, b_i, C_i, c_i, \widetilde C_i, \tilde{c}_i$ and $f_i$, where $A_i \in \mathbb{R}^{m \times d_i}$, $b_i \in \mathbb{R}^m$, $C_i \in \mathbb{R}^{p_i \times d_i}$, $c_i \in \mathbb{R}^{p_i}$, $\widetilde C_i \in \mathbb{R}^{\tilde p_i \times \tilde d}$, $\tilde c_i \in \mathbb{R}^{\tilde p_i}$ and $f_i$ is a convex function. Each node can perform matrix-vector multiplications with its matrices and compute function values and gradients of $f_i$.

The nodes (agents) are organized into a computational network $\cG = (\cV, \cE)$, where $\cV = \braces{1, \ldots, n}$ is the set of vertices and $\cE$ is the set of edges. Each agent can exchange information with those nodes to which it is connected by an edge. The constraints in \eqref{prob:opt_mixed_constraints_intro} can be divided into three types:

\textbf{Coupled Constraints.} These constraints link the variables across multiple nodes and require coordination among nodes to satisfy. The main coupled constraints are
\begin{equation}\label{eq:coupled_cons}
\sum_{i=1}^n (A_i x_i - b_i) = 0, %
\end{equation}
which jointly restrict the variables \( x_i \).

\textbf{Local Constraints.} These constraints depend only on the local variables and data at each node \(i\). Specifically, 
\begin{equation}
C_i x_i = c_i,
\end{equation}
which can be enforced independently by node \(i\) without coordination with others. They typically represent local restrictions or operational limits specific to each node.

\textbf{Shared Variable Constraints.} This type of constraint is imposed on the shared variable $\tilde x$. These constraints take the form
\begin{equation}
\widetilde C_i \tilde x = \tilde c_i,
\end{equation}
The variable $\tilde x$ is the same at all the nodes, but the constraint matrices $\widetilde C_i$ are individual.

If the problem includes a combination of coupled constraints, local and shared variable constraints, we call it a problem with \textit{mixed constraints}. 
Our study is largely motivated by classical work on distributed yet centralized algorithms \cite{boyd2011distributed}, which shows how playing with affine-constraint reformulations enable decomposition and distributed optimization of various problems.
Our main goal is to bring this level of flexibility to decentralized optimization, together with development of theoretical tools for comparing the complexity of different reformulations.

Problems of type \eqref{prob:opt_mixed_constraints_intro} have applications in several machine learning problems, including horizontal and vertical federated learning, distributed multi-task learning, and control of distributed energy systems.

\paragraph{Horizontal Federated Learning (HFL, or Consensus Optimization).}
Let the variables locally held by the nodes be connected by consensus constraints, i.e.,\ $x_1 = \ldots = x_n$. We come to the problem of consensus optimization.
\begin{equation}\label{prob:consensus_cons}
    \begin{aligned}
      \min_{x_1,\ldots,x_n}&~ \sum_{i=1}^n f_i(x_i) \\
      \text{s.t. }~ &x_1 = \ldots = x_n.
    \end{aligned}
\end{equation}
This is a standard decentralized optimization (or horizontal federated learning) problem statement \cite{boyd2011distributed,yang2019federated,kairouz2021advances,nedic2009distributed,scaman2017optimal,kovalev2021adom,kovalev2021lower,li2021accelerated}. The training data is distributed between computational entities by samples, i.e., each node holds a part of the dataset. During the optimization process, the agents aim to maintain equal model weights $x_i$.

Coupled constraints are capable of expressing consensus constraints: for instance, cyclic equalities $x_1 - x_2 = 0$, $x_2 - x_3 = 0, \ldots,$ $x_n - x_1 = 0$ can be written as $\sumin A_i x_i = 0$, where $x_i \in \R^d$, $A_1 = \pmat{I_d & 0_d & \cdots & 0_d & -I_d}^\top$ and so on.
It is interesting, though, that all black-box reductions of consensus constraints to coupled constraints are provably ineffective for first-order decentralized algorithms, as any choice of matrices $A_i$ increases communication complexity by at least a factor of $\sqrt{n}$ compared with optimal specialized consensus optimization algorithms \citep[Appendix A]{yarmoshik2024decentralized}. 
This fact motivates the explicit differentiation between consensus and coupled-constrained variables in \eqref{prob:opt_mixed_constraints_intro}.

\paragraph{Vertical Federated Learning (VFL)}
Unlike consensus optimization, where the data is distributed sample-wise between the nodes, in VFL the data is distributed feature-wise \cite{liu2024vertical,chen2020vafl}. Each node corresponds to a party possessing its local subset of weights $X_i$,
and submatrix of features $F_i$. 
In deep VFL each party transforms its local features into an intermediate representation, which is then consumed by a shared ``top'' model.
A simple instance keeps the party-side mapping linear, $H_i \coloneqq F_i X_i \in \R^{N\times m}$ with $X_i \in \R^{d_i\times m}$, aggregates $Z \coloneqq \sumin H_i$, and predicts $\hat y = g(Z;\widetilde X)$ with top-model parameters $\widetilde X$.
This yields the mixed-constraints formulation
\begin{align*}
  \min_{Z,\widetilde X, X_1,\ldots,X_n}~ &\ell(g(Z,\widetilde X)) + \sum_{i=1}^n r_i(X_i) + r(\widetilde X) \\
  \text{s.t. }~ &\sum_{i=1}^n F_i X_i = Z,
\end{align*}
where $r_i$, $r$ are regularizers.
Note that the affine constraint $\sumin F_{i} X_i - I Z = 0$ is a special case of coupled constraints.
Papers \cite{vepakomma2018split,xie2024improving} considered similar setup, but used nonlinear mappings for intermediate representations.

A fully decentralized implementation can avoid a dedicated ``top'' node by replicating the top-model parameters and splitting the sample dimension for computational parallelism.
Partition each party feature matrix into vertical blocks $F_i = \col\cbraces{F_{i,1},\ldots,F_{i,n}}$, which induces the corresponding split $H_i = \col\cbraces{F_{i,1}X_i,\ldots,F_{i,p}X_i}$.
Introduce local aggregated representations $Z_i$ and local copies of the top-model parameters $\widetilde X_i$.
Then VFL can be written as a problem with mixed coupled and consensus constraints that fits directly into \eqref{prob:opt_mixed_constraints_intro}:
\begin{align*}
  \min_{\substack{Z_1,\ldots,Z_n\\\widetilde X_1,\ldots,\widetilde X_n\\X_1,\ldots,X_n}}~ &\sum_{i=1}^n \ell_i\left(g(Z_i,\widetilde X_i)\right) + \sum_{i=1}^n r_i(X_i) + \frac1n\sum_{i=1}^n r(\widetilde X_i) \\
  \text{s.t. }~ &\sum_{i=1}^n F_{i,j}X_i = Z_j,\quad j=1,\ldots,n, \\
  &\widetilde X_1 = \ldots = \widetilde X_n,
\end{align*}
where $\ell_i$ aggregates the losses over the $i$-th sample block and the consensus constraint enforces a shared top model.

Other problem formulations that reduce to \eqref{prob:opt_mixed_constraints_intro} include a mixture of global and local models \cite{zhang2015deep,hanzely2020federated,hanzely2020lower}, distributed multi-task learning \cite{wang2016distributed} and federated self-supervised learning \cite{makhija2022federated}. We discuss these problems in Appendix \ref{app:problem_formulations}.

\paragraph{Related Work.}

Most popular special cases of problem \eqref{prob:opt_mixed_constraints_intro} are consensus constraints and coupled constraints. Both of these scenarios have been largely studied in the literature.

\textbf{Consensus constraints} have form $x_1 = \ldots = x_n$ and are a special case of \eqref{prob:opt_mixed_constraints_intro} when only the shared variable $\tilde x$ is present. Originating from the works \cite{nedic2009distributed,boyd2011distributed}, the work on decentralized consensus optimization came to building lower complexity bounds and optimal first-order algorithms \cite{scaman2017optimal,kovalev2020optimal} and generalization to time-varying graphs \cite{nedic2017achieving,kovalev2021adom,kovalev2021lower,li2021accelerated}. Generalizations on randomly varying networks were also done in \cite{koloskova2020unified,koloskova2021improved}. Algorithms for nonsmooth problems were proposed in \cite{scaman2018optimal,dvinskikh2019decentralized,gorbunov2019optimal,kovalev2024lower}.

Consensus optimization arises in large scale model training, i.e. in federated learning \cite{kairouz2019advances,mcmahan2017communication,lian2017can}.

\textbf{Coupled constraints} problem statements initially arised in control of distributed power systems. Different variants of method with corresponding engineering problem statements were studied in \cite{necoara2011parallel,necoara2014distributed, necoara2015linear}. A special case of coupled constraints is resource allocation problem \cite{doan2017distributed,li2018accelerated,nedic2018improved}. Aside from coupled equality constraints, nonlinear inequality constraints \cite{liang2019distributed,gong2023decentralized,wu2022distributed}, local constraints \cite{nedic2010constrained,zhu2011distributed}, restricted domains of local functions \cite{wang2022distributed,liang2019distributed,nedic2018improved,gong2023decentralized,zhang2021distributed, wu2022distributed} were studied. Generalizations on time-varying networks were presented in \cite{zhang2021distributed, nedic2018improved}.

The two main approaches to coupled constraints optimization include proximal ADMM-type algorithms and gradient methods. Proximal methods were studied in \cite{boyd2011distributed,chang2016proximal,falsone2020tracking,wu2022distributed}, and \cite{gong2023decentralized} studied an algorithm with inexact prox. Computing the proximal mapping is computationally tractable when the objective structure is prox-friendly. For this reason, proximal methods are mostly applied in power systems control, where the loss functions are simple enough. The situation is different in machine learning, where first-order methods are needed \cite{doan2017distributed,nedic2018improved,yarmoshik2024decentralized}. Lower complexity bounds and first-order optimal methods are presented in \cite{yarmoshik2024decentralized}.

The value of coupled constraints for machine learning research is mostly its application to vertical federated learning \cite{chen2020vafl,liu2024vertical,stanko2026accelerated}.

\textbf{Mixed constraints}. To the best of our knowledge, mixed constraints are little studied in the literature. The seminal work \cite{boyd2011distributed} proposes general form consensus, when each of the local variables $x_i$ is a subvector of the global variable $x$. However, if local functions do not depend on $x$, such constraint may be written in a coupled (not mixed) constraint form. Papers \cite{du2023linear,du2025distributed} proposes aggregative optimization, where each local function depends on the global decision vector $x$ along with global vector part $x_i$, and the optimization is performed w.r.t. coupled inequality constraints. They use a first-order method and achieve a linear convergence rate. Their problem can be written in form \eqref{prob:opt_mixed_constraints_intro}, but with local constraints that tie local individual variables to shared variables. In our paper, we study a slightly different setting.

\paragraph{Our Contributions.}

\begin{itemize}[leftmargin=*, noitemsep, topsep=0pt]
  \item A new class of decentralized problems with mixed affine constraints that generalizes most popular consensus-constrained and coupled-constrained setups, allowing for their efficient combination.
  \item Tight complexity analysis of first-order algorithms (with new corresponding lower bounds) in the smooth and strongly convex case, and supposedly near-optimal upper bounds for nonsmooth and non-strongly convex setups.
  \item A quantification of how joint structure of distributed affine constraints and communication network affects optimization performance.
\end{itemize}

\paragraph{Notation.}

For vectors $x_1, \ldots, x_n$ we denote a column-stacked vector $\bx = \col\cbraces{x_1, \ldots, x_n}$. We denote by $\norm{\cdot}$ the Euclidean norm, and by $\angles{\cdot, \cdot}$ the standard inner product. Sign $\otimes$ denotes the Kronecker product. We use bold letters for matrices stacked from different agents. These are block matrices and matrices obtained by Kronecker product, e.g. $\bA = \diag(A_1, \ldots, A_n)$, $\bA' = (A_1 \ldots A_n)$. Maximal and minimal positive singular values of matrix $A$ are denoted $\smax(A)$ and $\sminp(A)$, respectively. Maximal and minimal positive eigenvalues of a symmetric matrix $A$ are denoted $\lmax(A)$ and $\lminp(A)$, respectively.
For any matrix $A$, we introduce
$\label{eq:kappa_A}\kappa_A = \frac{\smax^2(A)}{\sminp^2(A)}.$ When referring to a group of matrices $A$, we interpret this parameter as defined for the block-diagonal matrix, i.e. $\kappa_A = \frac{\max_i{\smax^2(A_i)}}{\min_i{\sminp^2(A_i)}}$. For a given linear subspace $\cL$ we also introduce a corresponding projection operator $\bP_\cL$. 

\begin{table}[t]
	\vskip 0.15in
	\begin{center}
		\begin{small}
			\begin{sc}
				\begin{tabular}{ccc}
					\toprule
					Prob. & Oracle & Compl. \\
					\midrule
					\multirow{5}{*}{\vspace{0.5cm}\shortstack{Consensus \\ constr. \\ \eqref{prob:consensus_cons}}} 
					& Grad. & $\sqrt{\kappa_f}$ \\
					& Comm. & $\sqrt{\kappa_f}  \sqrt{\kappa_{W}}$ \\
					& Paper & \cite{scaman2017optimal} \\
					\midrule
					\multirow{5}{*}{\shortstack{Coupled \\ constr. \\ \eqref{eq:coupled_cons}}} 
					& Grad. & $\sqrt{\kappa_f}$ \\
					& Mat. $A$& $\sqrt{\kappa_f}\sqrt{\hat\kappa_{A}}$ \\
					& Comm. & $\sqrt{\kappa_f} \sqrt{\hat\kappa_{A}} \sqrt{\kappa_{W}}$ \\
					& Paper & \cite{yarmoshik2024decentralized} \\
					\midrule
					\multirow{4}{*}{\shortstack{Shared \\ var.\\constr. \\ \eqref{prob:shared_var_constraints}}}
					& Grad. & $\sqrt{\kappa_f}$ \\
					& Mat. $\widetilde C$& $\sqrt{\kappa_f}\sqrt{\hat\kappa_{\widetilde C^\top}}$ \\
					& Comm. & $\sqrt{\kappa_f}\sqrt{\hat\kappa_{\widetilde C^\top}} \sqrt{\kappa_W}$ \\
					& Paper & This paper, Th. \ref{thm:LN_upper_bound} \\
                    \midrule
					\multirow{5}{*}{\shortstack{Local \\ var. \\constr. \\ \eqref{prob:coupled_with_local}}}
					& Grad. & $\sqrt{\kappa_f}$ \\
					& Mat. $A$ & $\sqrt{\kappa_f} \sqrt{\kappaproj}$ \\
					& Mat. $C$ & $\sqrt{\kappa_f} \sqrt{\kappaproj}\sqrt{\kappa_{C}}$ \\
					& Comm. & $\sqrt{\kappa_f} \sqrt{\kappaproj} \sqrt{\kappa_W}$\\
					& Paper & This paper, Th. \ref{thm:mixed_upper_bound} \\
                    \midrule
                    \multirow{6}{*}{\shortstack{Mixed \\ constr. \\ \eqref{prob:opt_mixed_constraints_intro}}}
					& Grad. & $\sqrt{\kappa_f}$ \\
					& Mat. $A$ & $\sqrt{\kappa_f} \sqrt{\kappaproj} $ \\
					& Mat. $C$ & $\sqrt{\kappa_f} \sqrt{\kappaproj}\sqrt{\kappa_{C}}$ \\
                    & Mat. $\widetilde C$ & $\sqrt{\kappa_f}\sqrt{\hat\kappa_{\widetilde C^\top}}$ \\
					& Comm. & $\sqrt{\kappa_f}\left(\sqrt{\kappaproj} + \sqrt{\hkappa_{\widetilde C^\top}} \right)\sqrt{\kappa_W}$ \\
					& Paper & This paper, Th. \ref{thm:upper_bound_coupled_plus_non_ident_local_constraints_prob} \\
					\bottomrule
				\end{tabular}
			\end{sc}
		\end{small}
	\end{center}
	\caption{Convergence rates for decentralized smooth strongly convex optimization with affine constraints. Mixed condition numbers are denoted $\hat\kappa_A,~ \hat\kappa_{C^\top}$ (see Definition \ref{def:mixed_condition_number}). Condition number $\kappaproj$ is defined in Definition \ref{def:projected_condition_number}.
    The term $\log\left(\frac{1}{\varepsilon}\right)$ is omitted.}
    \vskip -0.3in    \label{table:smooth_str_convex_optimization_affine_constraints}
\end{table}

\paragraph{Paper Organization.} We begin with overview of optimization methods for convex optimization with affine constraints in Section~\ref{sec:affine_constraints_methods}. The algorithms described in Section~\ref{sec:affine_constraints_methods} are then applied to specific reformulations of corresponding special cases of problem \eqref{prob:opt_mixed_constraints_intro} in the following sections.
Section~\ref{sec:affine_constraints_statements} gives an overview of existing results in the field, and Section~\ref{sec:optimal_alg_Smooth_Str_Convex} contains complexity results (optimal algorithms with corresponding lower bounds) for smooth strongly convex setup which are summarized in \Cref{table:smooth_str_convex_optimization_affine_constraints}.
Analogously, the results for nonsmooth and non-strongly convex settings are described in Section~\ref{sec:nonsmooth_affine_constraints_statements}. The concluding remarks are given in Section~\ref{sec:conclusion}.

\section{Optimization with Affine Constraints} \label{sec:affine_constraints_methods}

All problems considered in this work can be formulated as convex optimization problems with affine equality constraints:
\begin{align}\label{prob:affine_constraints}
	\min_{u \in \cU} \; G(u) 
	\quad \text{s.t.} \quad 
	Bu = b ,
\end{align}
where $G \colon \cU \to \mathbb{R}$ is a convex function defined on a set $\cU \subseteq \mathbb{R}^d$, the matrix $B \in \mathbb{R}^{s \times d}$ is nonzero, and $b$ belongs to the column space of $B$. Depending on the structural properties of the objective function $G$, we employ different optimization techniques to solve problem~\eqref{prob:affine_constraints}.

\subsection{Assumptions on the Objective Function}

We introduce a set of assumptions that enable the derivation of convergence guarantees and complexity bounds for solving problem~\eqref{prob:affine_constraints} over several classes of convex functions.

\begin{assumption}[Strong convexity]\label{assum:strongly_convex}
	The function $G(u)$ is $\mu$-strongly convex on $\cU$ for $\mu \geq 0$, i.e. for any $u, u^\prime \in \cU$,
	\begin{align*}
		G(u^\prime)\geq G(u) + \angles{\nabla G(u), u^\prime - u} + \frac\mu2\normsq{u^\prime - u}.
	\end{align*}
	When $\mu = 0$, the function $G$ is said to be (non-strongly) convex.
\end{assumption}

\begin{assumption}[Smoothness]\label{assum:smooth}
	The function $G(u)$ is $L$-smooth on $\cU$, i.e. it is differentiable on $\cU$ and for any $u, u^\prime\in \cU$ we have
	\begin{align*}
		G(u^\prime) - G(u) - \angles{\nabla G(u), u^\prime - u}\leq \frac L2\normsq{u^\prime - u}.
	\end{align*}
\end{assumption}

\begin{assumption}[Bounded subgradients]\label{assum:bounded_gradient}
	The function $G$ has bounded subgradients on $\cU$, i.e. there exists a constant $M > 0$ such that for any $u \in \cU$ and any $G^\prime(u) \in \partial G(u)$,
	\begin{align*}
		\norm{G^\prime(u)}\leq M.
	\end{align*}
\end{assumption}

\subsection{Smooth Objectives}\label{sec:smooth_objectives_apapc}

Problems of type~\eqref{prob:affine_constraints} with smooth objectives (i.e. satisfying Assumption \ref{assum:smooth}) can be solved using the Accelerated Proximal Alternating Predictor-Corrector  algorithm (\textsc{APAPC}), proposed in \cite{salim2022optimal} for strongly convex problems. This is an optimal method for convex optimization with affine constraints when $\cU = \R^d$.
To apply the method for non-strongly convex objectives, we introduce a reduction technique via regularization which analysis is non-standard due to the presence of affine constraints. The theoretical complexity results for \textsc{APAPC} are summarized in \Cref{thm:smooth_objectives_apapc_convergence}.
\begin{theorem}\label{thm:smooth_objectives_apapc_convergence}
    Let $G(u)$ satisfy Assumptions \ref{assum:strongly_convex} and \ref{assum:smooth} with $0 < \mu < L$. Let $\norm{u^0 - u^*}^2\leq R^2$ and introduce accuracy $\eps > 0$.
    
    1. \textbf{Strongly convex case $\mu > 0$} \citep[Proposition 1]{salim2022optimal}.
    There exists a set of parameters for \Cref{alg:apapc} such that after $N = O\left(\kappa_B\sqrt{L/\mu} \log(1/\eps)\right)$ iterations the algorithm yields $u^N$ satisfying $\sqn{u^N - u^*}\leq \eps$.
    
    2. \textbf{Convex case $\mu = 0$} (\textbf{new,} Appendix~\ref{app:smooth_objectives_apapc}). Applying Algorithm \ref{alg:apapc} to regularized function $G^{\nu}(u) = G(u) + \nu/2 \|u^0 - u\|^2$ with $\nu = \eps/R^2$, after $N = O\left(\kappa_B\sqrt{LR^2/\varepsilon} \log \left(1 /  \eps\right)\right)$ iterations we obtain $u^N$, for which $G(u^N) - G(u^*) \leq \varepsilon$ and $\|B u^N - b\|^2 \leq O(\smax^2(B)\eps^2)$.
\end{theorem}

	\begin{algorithm}[h]
		\caption{\textsc{APAPC}}
		\label{alg:apapc}
		\begin{algorithmic}[1]
            \STATE {\bf Input:} $u^0 \in \R^d$
			\STATE {\bf Parameters:}  %
			$\eta,\theta,\alpha>0$, $\tau \in (0,1)$, $N \in \N$
			\STATE Set $u_f^0 = u^0$, $z^0 = 0 \in \R^d$
			\FOR{$k=0,1,2,\ldots, N-1$}{}
			\STATE $u_g^k \eqdef \tau u^k +  (1-\tau)u_f^k$\label{alg:apapc:line:x:1}
			\STATE $u^{k+\frac{1}{2}} \eqdef (1+\eta\alpha)^{-1}(u^k - \eta (\nabla G(u_g^k) - \alpha u_g^k + z^k))$\label{alg:apapc:line:x:2}
			\STATE $z^{k+1} \eqdef z^k  + \theta B^\top (B u^{k+\frac{1}{2}} - b)$ \label{alg:apapc:line:z}
			\STATE $u^{k+1} \eqdef (1+\eta\alpha)^{-1}(u^k - \eta (\nabla G(u_g^k) - \alpha  u_g^k + z^{k+1}))$\label{alg:apapc:line:x:3}
			\STATE $u_f^{k+1} \eqdef u_g^k + \tfrac{2\tau}{2-\tau}(u^{k+1} - u^k)$\label{alg:apapc:line:x:4}
			\ENDFOR
            \STATE {\bf Output:} $u^N$
		\end{algorithmic}
	\end{algorithm}

\subsection{Non-Smooth Objectives} \label{subsec:sliding}
In this section, we present an approach for solving Problem~\eqref{prob:affine_constraints} when the objective function $G$ is non-smooth. Following~\cite{dvinskikh2021decentralized, gorbunov2019optimal}, we handle the affine equality constraints by introducing the penalized objective function:
\begin{equation} \label{eq:affine_constraints_penal_prob}
	\min_{u \in \cU} \; H_{r}(u) := G(u) + \frac{r}{2} \normsq{B u - b},
\end{equation}
where $r>0$ is the penalty coefficient.

With a appropriate choice of $r$, solving the original problem reduces to minimizing $H_r$ (see \Cref{app:gradient_sliding_discuss} for details). Since $H_r$ consists of a non-smooth term $G(u)$ and a smooth quadratic penalty, we can apply the sliding technique to obtain separate complexity bounds for each component. Specifically, we employ the \textsc{Gradient Sliding} algorithm from \cite{Lan2019lectures}. To describe the algorithm, we define the following quadratic model of $H_r$:
\begin{align} \label{eq:sliding_subproblem}
    \Phi_{H_r}(w, u_1, u_2, &u_3, \beta, \eta) = \angles{G^\prime(u_1)  + r B^\top(Bu_2 -b), w} \notag \\
    &+ \frac{\beta}{2} \normsq{u_3 - w} + \frac{\beta \eta}{2} \normsq{u_1 - w},
\end{align}
where $\beta, \eta$ are parameters and $w, u_1, u_2, u_3 \in U$. We then derive the \Cref{alg:sliding} for solving Problem~\eqref{prob:affine_constraints} in non-smooth setting. At each iteration, the algorithm finds a solution of the subproblem of type \eqref{eq:sliding_subproblem}.

\begin{algorithm}[H]
    \caption{\textsc{Gradient Sliding}}
    \label{alg:sliding}
    \begin{algorithmic}[1]
        \STATE {\bf Input:} $u^0 \in U$
        \STATE {\bf Parameters:}  %
          $\{\gamma_k\} _{k=1}^\infty, \{\eta_k\}_{t=1}^\infty, \{\theta_t\}_{t=1}^\infty \subseteq \R_{++}$, 
          $N \in \N$, $\{T_k\}_{k=1}^N \subseteq \N^N$ 
        \STATE $\overline{u}^0 \eqdef u^0$
        \FOR{$k=1,2,\ldots, N$}{}
        \STATE $\underline{u}^k \eqdef \gamma_k u^{k-1} +  (1-\gamma_k)\overline{u}^{k-1}$\label{alg:sliding:line:x:1}
        \STATE $u^k_0 \eqdef \tilde{u}^k_0$, $\tilde{u}^k_0 \eqdef u^{k-1}$\label{alg:sliding:line:x:2}
        \FOR{$t=1,2,\ldots, T_k$}{}
        \STATE $u^k_t = \argmin\limits_{w \in \cU}\; \Phi_{H_r}(w, u^k_{t-1}, \underline{u}^{k}, u^{k-1}, \beta_k, \eta_t)$  \label{alg:sliding:line:x:3} 
        \STATE $\tilde{u}^k_t = \theta_t u^k_t + (1 - \theta_t) \tilde{u}^k_{t-1}$ \label{alg:sliding:line:x:4}
        \ENDFOR
        \STATE$u^k \eqdef u^k_{T_k}$, $\tilde{u}^k \eqdef \tilde{u}^k_{T_k}$ \label{alg:sliding:line:x:5}
        \STATE $\overline{u}^{k} \eqdef \gamma_k \tilde{u}^k +  (1-\gamma_k)\tilde{u}^{k-1}$\label{alg:sliding:line:x:6}
        \ENDFOR
        \STATE {\bf Output:} $\overline{u}^N$
    \end{algorithmic}
\end{algorithm}

\Cref{thm:upper_bound_nonsmooth_affine_constraints} provides upper bounds on the number of gradient computations and matrix multiplications required for \Cref{alg:sliding} to reach $(\varepsilon, \delta)$ accuracy, which means $G(u) - G(u^*) \leq \varepsilon$ and $\norm{Bu -b} \leq \delta$. For details refer to \Cref{app:gradient_sliding_discuss}.

\begin{theorem}[{\citet{lan2020communication}}]\label{thm:upper_bound_nonsmooth_affine_constraints}
	Let $G$ satisfy Assumption~\ref{assum:strongly_convex} with $\mu \geq 0$, Assumption~\ref{assum:bounded_gradient} with $M>0$, and suppose $\norm{u^0 - u^*}^2\leq R^2$. Given $\varepsilon>0$, consider solving the penalized problem~\eqref{eq:affine_constraints_penal_prob} using \Cref{alg:sliding}.\\
        1. \textbf{Convex case ($\mu = 0$)}. The algorithm requires $N_B = O\cbr{\frac{MR}{\eps} \sqrt{\kappa_B}}$ multiplications by $B, B^\top$ and $N_{G^\prime} = O\cbr{\frac{M^2R^2}{\eps^2} + N_B}$ evaluations of subgradient of $G$ to generate an output $\overline{u}^N$, for which $G(\overline{u}^N) - G(u^*) \leq \varepsilon$, and $\norm{B \overline{u}^N -b} \leq O(\varepsilon)$. \\
        2. \textbf{Strongly convex case ($\mu > 0$)}. With restarting, the algorithm returns $\overline u^N$ satisfying the same accuracy guarantees using $N_B = O\cbr{\frac{M}{\sqrt{\mu\eps}} \sqrt{\kappa_B}}$ multiplications by $B, B^\top$ and $N_{G^\prime} = O\cbr{\frac{M^2}{\mu\eps} + N_B}$ evaluations of subgradient of $G$.
\end{theorem}

\subsection{Chebyshev Acceleration} \label{subsec:chebyshev_acc}
Chebyshev acceleration \cite{salim2022optimal,scaman2017optimal,auzinger2011iterative}, initially applied as a theoretical tool to reduce communication complexity in smooth consensus optimization, effectively improves condition number of affine constraint matrix.
The idea of Chebyshev acceleration  is to equivalently reformulate an affine constraint $B u = b$ as $K x = b'$, where $K = P_B(B^\top B)$, $b' = \frac{P_B(B^\top B)}{B^\top B}B^\top b$, and $P_B$ is a polynomial such that  $P_B(\sigma_i^2(B)) = $ if and only if $\sigma_i(B) = 0$.
With appropriately scaled and shifted Chebyshev polynomials, the reformulation would have condition number $\kappa_{K} = O(1)$ and matrix-vector multiplications $K x$ could be implemented via $O(\sqrt{\kappa_B})$ matrix-vector multiplications by $B^\top u $ and $B u$, so there is no need to explicitly compute and store matrix polynomial $K$.
We will denote application of Chebyshev acceleration as $B \to P_B(B^\top B)$.
The pseudocode for the Chebyshev iteration is provided in Appendix~\ref{app:Chebyshev_iteration}.

\section{Existing Results in Decentralized Optimization with Affine Constraints} \label{sec:affine_constraints_statements}

This section discusses about known results in decentralized optimization and their convergence properties. We review related formulations of the decentralized optimization problem from prior work and their convergence properties. These formulations are special cases of the problem~\eqref{prob:opt_mixed_constraints_intro}.

\subsection{Notations and Assumptions}

Decentralized communication is typically represented as a matrix-vector multiplication. In particular, we assume the existence of a \emph{gossip matrix} $W \in \R^{n\times n}$, associated with the communication network $\cG = (\cV, \cE)$, with the following properties.
\begin{assumption}\label{ass:W}
    The \emph{gossip matrix} $W$ satisfies:\\
    \hspace*{0.5em} 1. $W$ is symmetric and positive semidefinite. \\
    \hspace*{0.5em} 2. $W_{ij}\neq 0$ if and only if $i = j$ or $(i,j) \in \cE$. \\
    \hspace*{0.5em} 3. $Wx = 0$ if and only if $x_1 = \ldots = x_n$.
\end{assumption}
An example of such matrix is the Laplacian matrix  $L = D - A$, where $A$ is the adjacency matrix and $D$ is the degree matrix of the network $\cG$. Throughout this paper, we denote $\bW = W\otimes I$, where the dimension of identity matrix $I$ is determined by the context.

We consider the following standard assumptions on the objective functions in decentralized optimization.
\begin{assumption} \label{ass:smooth_and_str_convex_of_f_i}
    Each function $f_i$ is $L_f$-smooth and $\mu_f$-strongly convex for some $L_f\geq \mu_f \geq 0$.
\end{assumption}
When $\mu_f >0$, we denote the condition number of the local functions by $\kappa_f = L_f / \mu_f$.

\begin{assumption} \label{ass:nonsmooth_and_str_convex_of_f_i}
    Each function $f_i$ has bounded subgradients with constant $M_f > 0$ and is $\mu_f$-strongly convex for some $\mu_f \geq 0$.
\end{assumption}

\subsection{Consensus Optimization}
This case corresponds to scenarios without affine equality constraints, involving only common variables. When Assumptions~\ref{ass:W} and~\ref{ass:smooth_and_str_convex_of_f_i} hold with $\mu_f > 0$, the optimal convergence rates are $O(\sqrt{\kappa_f} \log(1/\varepsilon))$ for gradient calls and $O(\sqrt{\kappa_f} \sqrt{ \kappa_W }  \log(1/\varepsilon))$ for communication rounds, as established in \cite{kovalev2020optimal}.

In the nonsmooth and non-strongly convex setting, that is, when Assumptions~\ref{ass:W} and~\ref{ass:nonsmooth_and_str_convex_of_f_i} hold with $\mu_f = 0$, the number of subgradient evaluations is upper bounded by $O(M_f^2R^2/\varepsilon)$, while the number of communication rounds is bounded by $O(\sqrt{\kappa_W} M_fR/\varepsilon)$, where $R$ denotes the radius of the feasible region. These bounds are achieved and shown to be optimal in \cite{scaman2018optimal}.

\subsection{Identical Local Constraints} \label{sub:ident_local}

Consider the case with only the common variable $\tilde{x}$ and identical local constraints, i.e., $\widetilde{C}_i = \widetilde{C}$ and $\tilde{c}_i = \tilde{c}$ for all $i = 1, \dots, n$. In \cite{rogozin2022decentralized_affine}, author obtained upper bounds for this formulation when Assumptions~\ref{ass:W} and~\ref{ass:smooth_and_str_convex_of_f_i} hold: $O(\sqrt{\kappa_f}\log\cbraces{1/\eps})$ gradient calls, $N_{\widetilde C} = O(\sqrt{\kappa_f} \sqrt{\kappa_{\widetilde C}} \log\cbraces{1/\eps})$ multiplications by $\widetilde C$ and $\widetilde{C}^{\top}$, and $\Nw = O\cbraces{\sqrt{\kappa_f} \sqrt{\kappa_W} \log\cbraces{1/\eps}}$ communications. As we will show in Section~\ref{sub:shared_variable_constraints}, the complexity bounds change substantially when local constraints are nonidentical.

\subsection{Coupled Constraints} 
\label{sub:coupled}

We now consider the case when each node locally holds a part of affine constraints $A_i$.
\begin{align} \label{eq:coupled_constraints_org}
	\min_{x_1\in\R^{d_1}, \ldots, x_n\in\R^{d_n}}~ \sumin f_i(x_i) \quad \text{s.t.} \quad \sumin \cbr{A_i x_i - b_i} = 0.
\end{align}

This problem was studied in \cite{yarmoshik2024decentralized} for smooth and strongly convex objective functions. Authors introduced a new kind of condition number that captured the convergence rates for coupled constraints.
\begin{definition}\label[definition]{def:mixed_condition_number}
    For a set of matrices $(B_1, \ldots, B_n)$, introduce an interaction matrix
    \begin{align}\label{eq:interaction_matrix}
       S_B = \frac{1}{n} \sum_{i=1}^n B_i B_i^\top,
    \end{align}
    and a mixed condition number
    \begin{align*}
        \hat\kappa_B = \frac{\underset{i=1,\ldots,n}{\max}\lmax(B_i B_i^\top)}{\lminp(S_B)} = \frac{\underset{i=1,\ldots,n}{\max}\lmax(B_i B_i^\top)}{\lminp\cbraces{\frac{1}{n}\sum_{i=1}^n B_i B_i^\top}}.
    \end{align*}
\end{definition}
\begin{remark}
    Note that the transposition order in definition of $\hat\kappa_B$ is important. Let
    \begin{align*}
        \hat\kappa_{B^\top} = \frac{\underset{i=1,\ldots,n}{\max}\lmax(B_i^\top B_i)}{\lminp(S_{B^\top})} = \frac{\underset{i=1,\ldots,n}{\max}\lmax(B_i^\top B_i)}{\lminp\cbraces{\frac{1}{n}\sum_{i=1}^n B_i^\top B_i}}.
    \end{align*}
    In general $\hat\kappa_{B^\top}\neq \hat\kappa_B$. For example, let $B_i = e_i$, where $e_i$ is the $i$-th coordinate vector. Then $\hat\kappa_B = n$ and $\hat\kappa_{B^\top} = 1$.
\end{remark}

\begin{remark}
    Note that in the case of equal matrices $B_1 = \ldots = B_n = B$ mixed condition number $\hkappa_B$ naturally reduces to usual condition number $\kappa_B$ so that $\hkappa_B = \hkappa_{B^\top} = \kappa_B$.
\end{remark}

When Assumptions~\ref{ass:W} and~\ref{ass:smooth_and_str_convex_of_f_i} hold with $\mu_f > 0$, the optimal convergence rates achieved in \cite{yarmoshik2024decentralized} are $O(\sqrt{\kappa_f} \log(1/\varepsilon))$ for gradient calls, $O(\sqrt{\kappa_f} \sqrt{ \hat{\kappa}_A }  \log(1/\varepsilon))$ matrix-vector multiplications $\bA$, $\bA^\top$ and $O(\sqrt{\kappa_f} \sqrt{ \hat{\kappa}_A } \sqrt{\kappa_W} \log(1/\varepsilon))$ for communication rounds.

\section{Optimal Algorithms for Smooth Strongly Convex Problems} \label{sec:optimal_alg_Smooth_Str_Convex}

Before deriving the algorithm and complexity bound for the problem~\eqref{prob:opt_mixed_constraints_intro}, we consider two special cases. The first is when only local constraints are imposed on the common variable, and the second is when both coupled and local constraints are present.

\subsection{Shared Variable Constraints}\label{sub:shared_variable_constraints}

We consider the case when each node is subject to its own local affine constraint on the global variable. In particular, we study the generation of formulation in \cite{rogozin2022decentralized_affine} from identical to non-identical constraints. The optimization problem can be written as
\begin{equation} \label{prob:shared_var_constraints}\tag{S}
   \min_{\tilde{x} \in \mathbb{R}^{\tilde{d}}} \sum_{i=1}^n f_i(\tilde{x})
   \quad \text{s.t.} \quad \widetilde{C}_i \tilde{x} = \tilde{c}_i \quad \forall i \in \{1,\ldots,n\}.
\end{equation}
Introducing local copies of the variable allows us to reformulate the problem in a block-matrix format
\begin{equation} \label{prob:local_non_ident_constraints_on_shared_var_prob} 
\hspace{-0.1cm}
   \min_{\tilde{\mathbf{x}} \in \left(\mathbb{R}^{\tilde{d}}\right)^n} F(\tilde{\mathbf{x}}) := \sum_{i=1}^n f_i(\tilde x_i)
   \;\textrm{ s.t. }\; \mathbf{W}\tilde{\mathbf{x}} = 0, \; \widetilde{\mathbf{C}}\mathbf{x} = \tilde \bc,
\end{equation}
where $\tilde{\mathbf{x}} = \col\{\tilde x_1, \ldots, \tilde x_n\}$, $\tilde{\mathbf{c}} = \col\{\tilde c_1, \ldots, \tilde c_n\}$, $\mathbf{W} = W \otimes I_d$ encodes consensus, and matrix $\widetilde{\mathbf{C}} = \diag(\widetilde C_1,\ldots, \widetilde C_n) \in \mathbb{R}^{\tilde p \times \tilde{d} n}$ collects the local constraints, where $\tilde{p} = \sum_{i=1}^n \tilde p_i$.

The constraint matrix has a block structure: $\widetilde{\mathbf{B}}^\top = (\widetilde{\mathbf{C}}^\top \;\;  \gamma \mathbf{W})$. Applying Lemma 2 from \cite{yarmoshik2024decentralized}, we select an appropriate scaling parameter $\gamma$ and derive the upper bounds for decentralized optimization with affine equality constraints on shared variable in \Cref{thm:LN_upper_bound}.

\begin{theorem}[\textbf{new,} Appendix \ref{app:nonind_local_lower}] \label{thm:LN_upper_bound}
Applying Algorithm~\ref{alg:apapc} to a reformulation of problem \eqref{prob:local_non_ident_constraints_on_shared_var_prob} with $\bW\to P_W(\bW)$ and then $\widetilde{\bB}\to P_C(\widetilde{\bB}^\top\widetilde{\bB})$, we obtain a method with complexity specified in \Cref{table:smooth_str_convex_optimization_affine_constraints}.
This bound is optimal in a naturally defined class of decentralized first-order algorithms for problems with local constraints.
\end{theorem}

A notable distinction from the coupled case is that regularization is not needed in order to guarantee strong convexity, and there is more freedom in the design of preconditioners. In particular, $\widetilde{\mathbf{C}}$ can be preconditioned independently of $\mathbf{W}$, whereas for coupled constraints $\mathbf{A}$ and $\mathbf{W}$ are inherently entangled. Indeed, $(\widetilde{\mathbf{C}}^\top \; \mathbf{W})^\top x = 0$ if and only if $(P(\widetilde{\mathbf{C}}^\top\widetilde{\mathbf{C}})^\top \; \mathbf{W})^\top x = 0$ for any polynomial $P$ without a constant term, provided that $P(\lambda)\neq 0$ for all nonzero eigenvalues $\lambda$ of $\widetilde{\bC}^\top \widetilde{\mathbf{C}}$. This property enables richer preconditioning strategies. Nevertheless, despite this additional flexibility, the upper bounds obtained above turn out to be tight. Optimality can be proved in a manner analogous to the lower bound arguments for coupled constraints (see Appendix~\ref{app:nonind_local_lower}).

\subsection{Coupled and Local Constraints}\label{sub:mixed_coupled_and_local}
Individual variables are related by coupled constraints and locally tied by local constraints.
\begin{align}\label{prob:coupled_with_local}
	\min_{x_1\in\R^{d_1}, \ldots, x_n\in\R^{d_n}}~ &\sumin f_i(x_i) \tag{C\&L} \\ \text{s.t.}~ &\sumin(A_i x_i - b_i) = 0,~ C_i x_i = d_i. \nonumber
\end{align}
For this formulation the decentralized-friendly matrix of the affine constraints is $\bB = \pmat{\bA & \alpha \bW\\ \beta\bC & 0}$ with $\bA = \diag\cbr{A_1, \ldots, A_n}$ and $\bC = \diag\cbr{C_1, \ldots, C_n}$, see Appendix~\ref{app:coupled_cons}. 

Now we introduce natural complexity parameters for this setup. 

\begin{definition}\label[definition]{def:projected_condition_number}
    Consider two sets of matrices $(B_1, \ldots, B_n)$ and $(D_1, \ldots, D_n)$. Let $\bB' = (B_1, \ldots, B_n)$ and $\bD = \diag(D_1, \ldots, D_n)$. We define 
    \begin{subequations}\label{eq:proj_condition_number}
    \begin{align*}
        \widetilde\mu_{BD} &:= \frac1n\sminp^2\cbr{\bB' \bP_{\ker\bD}}, \\
        \widetilde\kappa_{BD} &:=
         \begin{cases}
        \displaystyle
        \frac{\max\limits_{i=1,\ldots,n} \lmax(B_i B_i^\top)}
        {\widetilde{\mu}_{BD}}, 
        & \text{if } \widetilde{\mu}_{BD} > 0, \\[1ex]
        1, 
        & \text{if } \widetilde{\mu}_{BD} = 0.
        \end{cases}
    \end{align*}
     The definition of $\widetilde \kappa_{BD} = 1$ for the case $\widetilde \mu_{BD} = 0$ is introduced for convenience, since in that case $\tkappa_{BD}$ does not affect the complexity (see \Cref{lem:mixed_coupled_with_local_spectrum} and \Cref{thm:mixed_upper_bound}).
    
    \end{subequations}
\end{definition}
\begin{remark}\label[remark]{rem:projected_condition_number}
    Note that if $\widetilde\mu_{BD} > 0$ then we have $\tkappa_{BD}\leq \hkappa_B$. Indeed,   for any matrix $M$ and linear subspace $\cL$ of compatible dimensions it holds $\ker^\bot M \bP_\cL = \ker^\bot M \cap \cL$, and thus  $\sminp(M\bP_\cL) = \min_{h \in \ker^\bot M, h \in \cL} \|Mh\| /\|h\| \geq \min_{h \in \ker^\bot M}\|Mh\| / \|h\| =  \sminp(M)$.
    The equality is reached if a singular vector corresponding to $\sminp(\bB')$ belongs to $\ker\bD$. In particular, we have $\tkappa_{BD} = \hkappa_B$ when $\bD = 0$.
\end{remark}

Following lemma gives upper bound for $\kappa_B$ which is essential for complexity upper bounds. It also characterizes how interconnection between distributed affine constraints and communication network affects optimization performance.
\begin{lemma}[\textbf{new,} Appendix \ref{app:mixed_coupled_with_local_spectrum}] \label[lemma]{lem:mixed_coupled_with_local_spectrum}
  There exist constants $\alpha>0$ and $\beta>0$ such that the condition number
$\kappa_\bB$ satisfies
  \begin{equation}
    \kappa_\bB = \begin{cases}
      O\cbr{\kappaproj\kappa_W+ \kappaproj \kappa_C}, &\muproj > 0, \\
      O\cbr{\kappa_W + \kappa_C}, & \muproj = 0.
    \end{cases}
  \end{equation}
\end{lemma}

\begin{theorem}[\textbf{new,} Appendix \ref{app:mixed_coupled_with_local_lower}] \label{thm:mixed_upper_bound}
    Applying Algorithm~\ref{alg:apapc} to a reformulation of problem~\eqref{prob:coupled_with_local} with $\bC \to P_C(\bC^\top\bC)$, $\bW \to P_W(\bW)$ and then $\bB \to P_B(\bB^\top\bB)$ we obtain a method with complexity specified in \Cref{table:smooth_str_convex_optimization_affine_constraints}.
    These complexity bounds are optimal due to corresponding lower bounds.
 \end{theorem}

When there are no local constraints ($\bC = 0$) we have $\kappaproj = \hkappa_A$ (see Remark \ref{rem:projected_condition_number}).
In this case, the complexity bounds obtained in this subsection are identical to the bounds in \Cref{thm:coupled_iclr}, providing a strong generalization of the main results from \cite{yarmoshik2024decentralized}.

\subsection{Mixed Constraints}
We introduce $\bK = \diag(\bB, \widetilde{\bB})$, $\bz = \col(\bx, \by, \tilde{\bx})$,  $\bv = \col(\bb, \bc, \tilde{\bc},\mathbf{0})$ and define the feasible set
 $\mathcal{Z} = \R^d \times (\R^m)^n \times (\R^{\tilde{d}})^n$. Then, problem~\eqref{prob:opt_mixed_constraints_intro} can be reformulated as
\begin{equation} \label{eq:mixed_constraints_reformulation}
\min_{\bz \in \mathcal{Z}}
 \;G(\bz) \coloneqq  \sum_{i=1}^n f_i(x_i, \tilde{x}_i) \quad
\text{s.t.}\quad
 \bK \bz = \bv.
\end{equation}
We now gather all our previous results and illustrate how complexity bounds for decentralized optimization with shared variable constraints (Section \ref{sub:shared_variable_constraints}) and with coupled and local constraints (Section \ref{sub:mixed_coupled_and_local}) are combined in the general mixed setting formulated in \eqref{prob:opt_mixed_constraints_intro}.
\Cref{thm:upper_bound_coupled_plus_non_ident_local_constraints_prob} gives upper bounds for problem~\eqref{prob:opt_mixed_constraints_intro}.

\begin{theorem}[\textbf{new,} Appendix~\ref{app:upper_bound_coupled_plus_non_ident_local_constraints_prob}] \label{thm:upper_bound_coupled_plus_non_ident_local_constraints_prob}
    Applying Algorithm~\ref{alg:apapc} to problem~\eqref{eq:mixed_constraints_reformulation}, we obtain a method with complexity specified in \Cref{table:smooth_str_convex_optimization_affine_constraints}.
    In the case of identical local constraints when matrices $\widetilde C_i$ are equal, complexities $N_{\widetilde C}$ and $N_W$ change to 
     \begin{align*}
       N_{\widetilde C} &= O\cbraces{\sqrt{\kappa_f} \sqrt{\kappa_{\widetilde C}} \log\cbraces{\frac1\eps}} \\
         N_W &= O\cbraces{\sqrt{\kappa_f} \sqrt{\kappaproj} \sqrt{\kappa_W} \log\cbraces{\frac1\eps}}.
     \end{align*}
\end{theorem}

Thus, in the presence of both coupled and local constraints acting on different sets of variables, our unified framework yields tight complexity bounds. These results interpolate between the purely coupled and purely local cases, showing that the overall difficulty is governed by the joint conditioning of $\bA$, $\bC$, and the network topology $W$.
These bounds are optimal due to combination of lower bounds in \Cref{thm:LI_upper_bound,thm:coupled_iclr} in case of identical local constraints and \Cref{thm:LN_upper_bound,thm:coupled_iclr} in the general case.

\section{Extension to Non-Strongly Convex and Nonsmooth Optimization}\label{sec:nonsmooth_affine_constraints_statements}

\subsection{Smooth and Convex Optimization} \label{sec:smooth_convex_case}

Using \Cref{thm:smooth_objectives_apapc_convergence} in Section~\ref{sec:affine_constraints_methods}, we derive the upper bounds
on the iteration complexity for problem~\eqref{prob:opt_mixed_constraints_intro} in smooth, non-strongly convex regime, as stated in \Cref{thm:upper_bounds_mixed_smooth_nonstrongly_convex}.

\begin{theorem}[\textbf{new,} Appendix~\ref{app:smooth_convex_results}] \label{thm:upper_bounds_mixed_smooth_nonstrongly_convex}
     Let Assumptions~\ref{ass:W} and~\ref{ass:smooth_and_str_convex_of_f_i} hold with $L_f > \mu_f = 0$. Consider applying Algorithm~\ref{alg:apapc} with a regularization technique to problem~\eqref{eq:mixed_constraints_reformulation}, which is a reformulation of problem~\eqref{prob:opt_mixed_constraints_intro}. Then the resulting method has complexity specified in Table~\ref{table:smooth_conv_optimization_constraints}.
\end{theorem}

From this Theorem, we directly obtain the results for other types of constraints, they are summarized in Table~\ref{table:smooth_conv_optimization_constraints}. See Appendix~\ref{app:smooth_convex_results} for explanations. Note that these upper bounds are suboptimal due to the logarithmic factor.

\subsection{Nonsmooth and Convex Optimization} \label{sec:nonsmooth_convex_case}

Analogically, we derived the results for Nonsmooth and Convex case (Table~\ref{table:nonsmooth_conv_optimization_constraints}).

\begin{theorem}[\textbf{new,} Appendix~\ref{app:nonsmooth_convex_results}] \label{thm:nonsmooth_convex_upper_bounds}
    Let Assumptions \ref{ass:nonsmooth_and_str_convex_of_f_i} hold with $\mu_f = 0$ and $M_f > 0$.
    Applying Gradient Sliding to problem~\eqref{eq:mixed_constraints_reformulation} with Chebyshev Acceleration we obtain a method 
    with complexity specified in \Cref{table:nonsmooth_conv_optimization_constraints}.
\end{theorem}

\subsection{Nonsmooth and Strongly Convex Optimization} \label{sec:nonsmooth_strongly_convex_case}

In the strongly convex and non-smooth setting, we consider the problem~\eqref{prob:opt_mixed_constraints_intro} on a bounded set $X_1 \times \dots \times X_n \times \widetilde{X}$,  otherwise Assumption~\ref{ass:nonsmooth_and_str_convex_of_f_i} with $\mu_f > 0$ cannot be held. Denote $\mathcal{X} = X_1 \times \dots \times X_n$, $\widetilde{\mathcal{X}} = \left(\widetilde{X}\right)^n$. Lemma~\ref{lem:strong_conv_G_penal} ensures that the penalized reformulation of $G$ is strongly convex.

\begin{lemma}[\textbf{new,} Appendix~\ref{app:proof_strong_conv_G_penal}] \label{lem:strong_conv_G_penal}
    Suppose that Assumption~\ref{assum:strongly_convex} holds with $\mu_f > 0 $. Let $\alpha$ and $\varepsilon$ satisfy following conditions:
    \begin{equation} \label{eq:strcv_non_smooth_coupled_constraints_constants_condition_main}
        \alpha^2 = \frac{\mu_{\bA} + L_{\bA}}{\mu_{\bW}}, \quad \varepsilon \leq \frac{4 r^2 \mu_\bA}{\mu_f}.
    \end{equation}
    Then, the penalized function of $G(\bx, \by, \tilde{\bx})$, where $G$ is defined in \eqref{eq:mixed_constraints_reformulation}, is $\frac{\mu_f}{2}$-strongly convex on $\mathcal{X} \times \mathcal{L}_m^\perp \times \widetilde{\mathcal{X}}$.
\end{lemma}

By integrating the results from Section~\ref{subsec:sliding}, Section~\ref{sec:affine_constraints_statements} and Lemma~\ref{lem:strong_conv_G_penal}, we establish upper bounds on the iteration complexity of decentralized optimization with mixed affine constraints. The summarization of results is presented in Table~\ref{table:nonsmooth_strconv_optimization_constraints}. 

\begin{theorem}[\textbf{new}, Appendix~\ref{app:proof_nonsmooth_strongly_convex_upper_bounds}] \label{thm:nonsmooth_strongly_convex_upper_bounds}
     Let Assumptions \ref{ass:nonsmooth_and_str_convex_of_f_i} hold with $\mu_f > 0$ and $M_f > 0$.
    Applying Restarted Gradient Sliding with Chebyshev Acceleration we obtain a method 
    with complexity specified in \Cref{table:nonsmooth_strconv_optimization_constraints}.
\end{theorem}

\section{Conclusion}\label{sec:conclusion}

We unify a series of results in decentralized optimization under a problem statement of mixed affine constraints. The generality of our formulation reduces to horizontal and vertical federated learning, control of distributed systems and other topics in distributed machine learning. Our aim is to analyze the problem systematically. To do this, we prove lower complexity bounds and provide corresponding optimal methods for different variants of our problem statement.

The logical continuation of the paper is the discussion of how different problem classes reduce to each other (see i.e. results in Table \ref{table:smooth_str_convex_optimization_affine_constraints}). Can the complexity of one optimization problem be reduced if the problem is rewritten in a different form?

A different future direction is to consider constraints $\sumin (A_i x_i + \tilde A_i \tilde x_i - b_i) = 0$. We hypothesize that such constraints can cover new machine learning formulations, i.e. mixture of global and local models in federated learning and distributed self-supervised learning.

\bibliography{references}
\bibliographystyle{icml2026}

\newpage
\appendix
\onecolumn

\section{Additional notation}

By $\cL_m$ we denote the so-called consensus space, which is given as $\cL_m = \{(y_1,\ldots,y_n) \in (\R^m)^n : y_1 , \ldots , y_n \in \R^m\; \text{ and }\; y_1 = \cdots = y_n\}$, and $\cL_m^\perp$ denotes the orthogonal complement to $\cL_m$, which is given as
\begin{equation}\label{eq:Lperp}
	\cL_m^\perp = \{(y_1,\ldots,y_n) \in (\R^m)^n : y_1 , \ldots , y_n \in \R^m\;\text{ and }\; y_1 + \cdots + y_n = 0\}.
\end{equation}

If not otherwise specified, for any matrix $A$ we denote by $L_A$ and $\mu_A$ some upper and lower bound on its maximal and minimal positive squared singular values respectively:
\begin{equation}\label{eq:mat_Lmu}
  \lmax(A^\top A) = \smax^2(A) \leq L_{A},
    \qquad
	\mu_{A} \leq \sminp^2(A) = \lminp(A^\top A).
\end{equation}
When referring to a group of matrices $A = (A_1\ldots A_n)$, we interpret these parameters as defined for the block-diagonal matrix, e.g. $L_A = \max_i{\smax^2(A_i)}$, $\mu_A = \min_i{\sminp^2(A_i)}$.

\newpage

\section{Problem Formulations for Mixed Constraints}\label{app:problem_formulations}

\paragraph{Mixture of Local and Global Models.}

The difference with consensus optimization is that locally held model weights may change from one node to another.
\begin{align*}
	\min_{\substack{z, x_1, \ldots, x_n}}~ &\sumin f_i(x_i) + \frac\lambda2 \norm{x_i - z}^2 \\
	\text{s.t. } &x_1 + \cdots + x_n = nz %
\end{align*}
The variable $z$ is stored at the first node, and we come to a problem with coupled constraints.

\paragraph{Distributed Multi-Task Learning (MTL).}

In \cite{wang2016distributed}, the authors propose a distributed MTL. Every node trains its own model, but joint (non-separable) regularizers are enforced in a per-coordinate fashion.
\begin{align*}
	\min_{x_1, \ldots, x_n} \sumin f_i(x_i) + \lambda \sum_{j=1}^d r(x^{(j)}),
\end{align*}
where $x^{(j)}$ is a vector in $\R^n$ that consists of the $j$-th coordinates of $x_1, \ldots, x_n$. The regularizer is chosen as $r(x) = \norm{x}_2$ or $r(x) = \norm{x}_\infty$. This problem enables rewriting in a coupled constraints form. Introduce matrices $Q_{ij}$ of size $n\times d$ that have all zero entries but one unit entry at position $(i, j)$. Multiplication $y = Q_{ij} x$ returns a vector $y\in\R^n$ with $i$-th entry equal to the $j$-th entry of $x$ and all other components equal to zero.
\begin{align*}
	\min_{\substack{x_1, \ldots, x_n\\y_1, \ldots y_d}}~ &\sumin f_i(x_i) + \lambda \sum_{j=1}^d r(y_j) \\
	\text{s.t. } &y_j = \sumin Q_{ij} x_i.
\end{align*}
Here $y_j = x^{(j)}$ are vectors that collect $j$-th components of locally held $x_i$.

\textbf{Federated Self-Supervised Learning (SSL).}

In distributed SSL a group of agents seek to learn representations of the locally held data while synchronizing with others if the corresponding data has common features or samples. The parties share a common representation alignment dataset to align their representations \cite{makhija2022federated}. The corresponding problem can be formulated as optimization with consensus constraints. Let each node hold a model with weights $x_i\in\R^{d_i}$ (local models possibly have different architectures). For simplicity we assume that models are linear and their loss functions write as $\norm{C_i x_i - c_i}^2$. The consensus is enforced on inner representations, which we assume to have form $A_i x_i$. The problem of federated SSL writes as
\begin{align*}
    \min_{\substack{x_1, \ldots, x_n\\z_1, \ldots, z_n}}~ &\sumin \norm{C_i x_i - c_i}^2 + \frac\mu2 \norm{A_i x_i - z_i}^2 \\
    \text{s.t. } &z_1 = \ldots = z_n.
\end{align*}
We see that this problem comes down to consensus optimization.

\newpage

\newpage

\section{Chebyshev Iteration} \label{app:Chebyshev_iteration}

\begin{algorithm}[H]
\caption{Chebyshev$(v, B, b)$ \cite{gutknecht2002chebyshev}}
\label{alg:chebyshev}
\begin{algorithmic}[1]
\STATE \textbf{Input:} $v, B, b$
\STATE $n \gets \left\lceil \sqrt{\frac{\smax^2(B)}{\sminp^2(B)}} \right\rceil$
\STATE $\rho \gets \left(\smax^2(B) - \sminp^2(B)\right)^2 / 16$, \quad $\nu \gets \left(\smax^2(B) + \sminp^2(B)\right)/2$
\STATE $\delta^0 \gets -\nu/2$
\STATE $p^0 \gets - B^{\top}(B v - b)/\nu$
\STATE $v^1 \gets v + p^0$
\FOR{$i = 1, \ldots, n-1$}
    \STATE $\beta^{i-1} \gets \rho / \delta^{i-1}$
    \STATE $\delta^i \gets -(\nu + \beta^{i-1})$
    \STATE $p^i \gets \big(B^{\top}(B v^i - b) + \beta^{i-1} p^{i-1}\big)/\delta^i$
    \STATE $v^{i+1} \gets v^i + p^i$
\ENDFOR
\STATE \textbf{Output:} $v^n$
\end{algorithmic}
\end{algorithm}

\newpage

\section{Proof of \Cref{thm:smooth_objectives_apapc_convergence}}\label{app:smooth_objectives_apapc}

\subsection{Strongly convex case ($\mu > 0$)}

To prove this case, we just recall a theorem on APAPC convergence from \cite{salim2022optimal}.
\begin{theorem}\label{thm:apapc}(\cite{salim2022optimal}). 
	Let Assumptions \ref{assum:strongly_convex} and \ref{assum:smooth} hold. There exists a set of parameters for \Cref{alg:apapc} such that to yield $u^N$ satisfying $\sqn{u^N - u^*}\leq \eps$ it requires $N = O(\kappa_B\sqrt{L/\mu} \log(1/\eps))$ iterations.
\end{theorem}

\subsection{Convex case ($\mu = 0$)}

Regularization can be used to adapt methods designed for strongly convex objectives to non-strongly convex problems. Let us define the regularized function as follows:
\begin{equation} \label{eq:reg_func}
	G^{\nu}(u) = G(u) + \frac{\nu}{2} \|u^0 - u\|^2.
\end{equation}
The following lemma describes the accuracy needed for solution of the regularized problem.

\begin{lemma}[\textbf{new}] \label[lemma]{lem:regularized_convergence}
	Let $G: \mathbb{R}^d \to \mathbb{R}$ be convex and $L$-smooth function and suppose that there exists solution $u^*\in \Argmin\limits_{u: B u = b} G(u)$ and $u_{\nu}^* =  \argmin\limits_{u: B u = b} G^{\nu}(u)$. Define $D = G(u^*) - \min\limits_{u} G(u)$. Assume that $\|u^0 - u^*\|^2 \leq R^2$. Recall $\nu = \eps/R^2$ from \Cref{thm:smooth_objectives_apapc_convergence} and set 
	\begin{align}\label{eq:regularized_convergence_def_delta}
		\delta =  \frac{\varepsilon^2}{32\left(D + \frac{\varepsilon}{2}\right) \left(L+\frac{\varepsilon}{R^2}\right)}.
	\end{align}
	If we have $\norm{u - u_\nu^*}^2\leq \delta$, then
	\begin{equation}
		G(u) - G(u^*) \leq \varepsilon,~ \norm{Bu - b}^2\leq \delta\smax^2(B).
	\end{equation}
	In other words, it is sufficient to solve the regularized problem with accuracy $\delta = O(\eps^2)$ in terms of convergence in argument to get an $\eps$-solution of the initial problem.
\end{lemma}
\begin{proof}
    We have
\begin{align*}
	G(u) - G(u^*) &\aleq{follows the definition of the regularized function $G^\nu$}  G^{\nu}(u) - G^{\nu}(u_{\nu}^*) + \frac{\nu}{2} \|u^0-u^*\|^2 \\
	&\aleq{uses the $(L+\nu)$-smoothness of $G^\nu$}\langle \nabla G^{\nu}(u_\nu^*), u - u_\nu^* \rangle + \frac{L+\nu}{2} \|u-u_\nu^*\|^2 + \frac{\nu}{2} \|u^0-u^*\|^2 \\
	&\aleq{uses the Cauchy–Schwarz inequality} \| \nabla G^{\nu} (u^*_{\nu}) \| \cdot \|u - u_\nu^*\| + \frac{L+\nu}{2} \|u-u_\nu^*\|^2 + \frac{\nu}{2} \|u^0-u^*\|^2 \\
	&\aleq{uses the convexity and smoothness of $G^\nu$}\sqrt{ 2(L+\nu) \left(G^{\nu}(u^*_\nu) - \min_{u} {G^{\nu}(u)}\right) \|u - u_\nu^*\|^2}  + \frac{L+\nu}{2} \|u-u_\nu^*\|^2 + \frac{\nu}{2} \|u^0-u^*\|^2 \\
	&\aleq{uses the fact that $G(u) \leq G^\nu(u)$ for all $u$}\sqrt{ 2(L+\nu) \left(G^{\nu}(u^*) - \min_{u} {G(u)}\right) \|u - u_\nu^*\|^2} + \frac{L+\nu}{2} \|u-u_\nu^*\|^2 + \frac{\nu}{2} \|u^0-u^*\|^2 \\
	&\aeq{uses the definition of $G^\nu$}\sqrt{ 2(L+\nu) \left(G(u^*) - \min_{u} G(u) + \frac{\nu}{2} \|u^0-u^*\|^2\right) \|u - u_\nu^*\|^2} + \frac{L+\nu}{2} \|u-u_\nu^*\|^2 + \frac{\nu}{2} \|u^0-u^*\|^2 \\
	&\aeq{uses the definitions of $R$, $D$ and the assumption that $\|u - u^*_\nu\| \leq \delta$}\sqrt{ 2 (L+\nu) \left(D + \frac{\nu R^2}{2}\right)\delta} + \frac{(L+\nu)\delta}{2} + \frac{\nu R^2}{2} \\
	&\aleq{is due to the definition of $\nu$}\sqrt{ 2 (L+\nu) \left(D + \frac{\varepsilon}{2}\right)\delta} + \frac{(L+\nu)\delta}{2} + \frac{\nu R^2}{2} \\
	&\aleq{uses the definition of $\delta$ in \eqref{eq:regularized_convergence_def_delta}} \frac{\varepsilon}{64} + \frac{\varepsilon}{4} + \frac{\varepsilon}{2} \\
	&< \varepsilon.
\end{align*}
where \annotate.

Moreover,
\begin{align*}
	\|Bu - b\|^2 & = \|B u - B u^*_\nu\|^2 \leq \smax^2(B) \cdot \|u - u^*_\nu\|^2 \leq \sigma_{\max}^2(B) \cdot \delta = O(\smax^2(B)\varepsilon^2).
\end{align*}
\end{proof}

From Proposition 1 in \cite{salim2022optimal} and Lemma \ref{lem:regularized_convergence}, to achieve $G(u^k) -G(u^*) \leq \varepsilon$, APAPC (Algorithm \ref{alg:apapc}) requires $\mathcal{O}\left(\kappa_B\sqrt{\frac{L+\nu}{\nu}} \log\left(\frac{1}{\delta}\right)\right)$ iterations.

From the definitions of $\delta$ in \eqref{eq:regularized_convergence_def_delta} and $\nu$ in Lemma~\ref{lem:regularized_convergence}, we have:
\begin{align*}
	\frac{1}{\delta} = \mathcal{O}\left(\frac{\left(M + \frac{\varepsilon}{2}\right) \left(L+\frac{\varepsilon}{R^2}\right)}{\varepsilon^2}\right),~~
	\frac{L+\nu}{\nu} = 1 + \frac{L}{\nu} = 1 + \frac{LR^2}{\varepsilon} = \mathcal{O}\left(\frac{LR^2}{\varepsilon}\right)
\end{align*}
Hence, the resulting complexity is $\mathcal{O}\left(\kappa_B\sqrt{\frac{LR^2}{\varepsilon}} \log \left( \frac{\left(M + \frac{\varepsilon}{2}\right) \left(L+\frac{\varepsilon}{R^2}\right)}{\varepsilon^2} \right)\right)$ iterations.

\newpage
\section{Discussion on Gradient Sliding Method}  \label[app]{app:gradient_sliding_discuss}

\subsection{Preliminary: Gradient Sliding}
Let us discuss the deterministic gradient sliding (GS) method \citep[Algorithm 8.1]{lan2020first};   
\cite{lan2020communication} that is used for optimization problems consisting of two summands to split the complexities. Consider problem
\begin{align}\label{prob:sliding}
	\min_{u \in U} K(u) := K_s(u) + K_n(u).
\end{align}

\begin{theorem}[\citet{lan2020first}]\label{thm:sliding}
	Let $K_s$ satisfy Assumption~\ref{assum:strongly_convex} with strong convexity parameter $\mu = 0$ and Assumption~\ref{assum:smooth}, $K_n(u)$ satisfy Assumption~\ref{assum:strongly_convex} with $\mu = 0$ and Assumption \ref{assum:bounded_gradient} and $\norm{u^0 - u^*}^2\leq R^2$. Gradient sliding algorithm applied to problem \eqref{prob:sliding} requires $N_s = O\cbr{\sqrt{\frac{LR^2}{\eps}}}$ calls to gradient of $K_s$ and 
	$N_n = O\cbr{\frac{M^2R^2}{\eps^2} + N_s}$
	calls to subgradient of $K_n$ to yield $\hat u$ such that $K(\hat u) - K^*\leq \eps$.
\end{theorem}

By the method of Lagrange
multipliers, problem \eqref{prob:affine_constraints} can be equivalently written as the following saddle point problem:
\begin{equation} \label{eq:saddle_point_problem}
	\begin{aligned}
		\min_{u \in U} \; \max_{v \in\R^p} \;\sbraces{G(u) +  \langle v, B u -b \rangle} 
	\end{aligned}
\end{equation}

\begin{lemma}[\citet{lan2020communication}]  \label[lemma]{lem:boundness_of_penal_prob_dual}
    Let $u^*$ be an optimal solution of \eqref{prob:affine_constraints}. Then there exists an optimal dual multiplier $v^*$ for \eqref{eq:saddle_point_problem} such that
	\begin{equation} \label{eq:optimal_dual_multipler_bound}
		\norm{v^*} \leq R_{\operatorname{dual}} \eqdef \frac{M}{\sminp(B)}.
	\end{equation}
\end{lemma}

\begin{lemma}[\citet{gorbunov2019optimal}] \label[lemma]{lem:penalization}
	Let $\varepsilon > 0$ and $r = \frac{2 R^2_{\operatorname{dual}}}{\varepsilon}$. If $\hat{u}$ is an $\varepsilon$-solution of \eqref{eq:affine_constraints_penal_prob}, i.e.
	\begin{equation*}
		H_{r} (\hat u) - \min_{u \in U}\; H_{r}(u) \leq \varepsilon,
	\end{equation*}
	then
	\begin{equation*}
		G(\hat u) - \min_{B u - b = 0} G(u) \leq \varepsilon, \quad \|B\hat{u} - b\| \leq \frac{2\varepsilon}{R_{\operatorname{dual}}}.
	\end{equation*}
\end{lemma}
Although \cite{lan2020communication} and \cite{gorbunov2019optimal} considered only consensus optimization ($B = \bW$, $b = 0$), proofs of \Cref{lem:boundness_of_penal_prob_dual,lem:penalization} in referenced sources work without changes for any $B$ and $b \in \image B$.

\subsection{Proof of \Cref{thm:upper_bound_nonsmooth_affine_constraints}}  \label[app]{app:affine_constraint_nonsmooth}

\begin{proof}
    Denote $K_n(u) = G(u)$ and $K_s(u) = \frac{R^2_{\operatorname{dual}}}\eps\sqn{B u}$, implying 
    $L = \frac{2R^2_{\operatorname{dual}} \smax^2(B)}{\eps}  \leq \frac{2M^2\smax^2(B)}{\eps\sminp^2(B)} = \frac{2M^2}{\eps}\kappa_B$, where the inequality follows from \Cref{lem:boundness_of_penal_prob_dual}. Applying \Cref{thm:sliding} gives the desired result.
    
    When $\mu>0$, we simply substitute $L = \frac{2M^2}{\eps}\kappa_B$ into the complexity of R-Sliding \citep[Theorem 8.3]{lan2020first} $N_s = O\cbraces{\sqrt{\frac{L}{\mu}}\log{\frac{1}{\eps}}}$, $N_n=O(\frac{M^2}{\mu \eps} + N_s)$.
    Note, that for some reason \citep[Section 8.1.3.1]{lan2020first} only considers the case then the smooth component is strongly convex.
    But as it can be seen from the proof of \citep[Theorem 8.3]{lan2020first}, it only uses strong convexity of the sum $K(u)$ of the smooth and the nonsmooth terms, thus we can apply it in our case where $K_n(u)$ is strongly convex, as was also done in \cite{dvinskikh2019decentralized,uribe2020dual}.
\end{proof}

\newpage

\section{Auxiliary Theorems and Lemmas for  Section~\ref{sec:affine_constraints_statements}}

\subsection{Identical Local Constraints}
Let the constraint matrices $C_i$ be equal. Then the block-diagonal matrix of affine constraints has form $\mathbf{C} = I_n \otimes C$. The spectral properties of $\bB^\top = [\bC^\top~~ \gamma\bW]$ can be revisited in comparison with Lemma~\ref{lem:local_spectrum}.

\begin{lemma}[{\citep[Lemma 1]{rogozin2022decentralized_affine}}]\label[lemma]{lem:local_spectrum}
	\begin{align*}
		\smax^2(\bB) &= \smax^2(C) + \gamma^2 \smax^2(W),
		\\
		\sminp^2(\bB) &= \min\br{\sminp^2(C), \gamma^2 \sminp^2(W)}.
	\end{align*}
\end{lemma}
As a result, we obtain an enhanced bound on the number of communications w.r.t. Theorem~\ref{thm:LN_upper_bound}.
\begin{theorem}[\cite{rogozin2022decentralized_affine}]
\label{thm:LI_upper_bound}
	Applying Algorithm~\ref{alg:apapc} to problem~\eqref{prob:local_non_ident_constraints_on_shared_var_prob} with $C_1 = \ldots = C_n = C$, after applying $\bW\to P_W(\bW)$ and $\bC\to P_C(\bC^\top\bC)$, we obtain a method that requires
	\begin{align*}
    \hspace{-0.5cm}
		\Nnabla &= O\cbraces{\sqrt{\kappa_f}\log\cbraces{\frac1\eps}} \text{gradient calls}, \\
		\Nc &= O\cbraces{\sqrt{\kappa_f} \sqrt{\kappa_C} \log\cbraces{\frac1\eps}} \text{mul. by } \bC \text{ and } \bC^\top, \\
		\Nw &= O\cbraces{\sqrt{\kappa_f} \sqrt{\kappa_W} \log\cbraces{\frac1\eps}} \text{communications}.
	\end{align*}
  This upper bounds are optimal due to corresponding lower bounds.
\end{theorem}

\subsection{Coupled Constraints}\label{app:coupled_cons}

Let $A_i \in \R^{m \times d_i}$.
A decentralized-friendly reformulation of coupled constraints is $\bA \bx + \bW \by = \bb$, where $\by \in \R^{mn}$. Then the problem writes as
\begin{align}\label{prob:coupled_constraints_prob}\tag{C}
	\min_{\bx\in\R^d}~ F(\bx) := \sum_{i=1}^n f_i(x_i) \quad \text{s.t.} \quad \bA\bx + \bW\by = \bb.
\end{align}
Indeed, since range of $\mathbf W$ is $\{\mathbf x : x_1 + x_2 + \ldots x_n = 0\}$ we have $\mathbf A \mathbf x - \mathbf b \in \text{range } \mathbf W \Leftrightarrow \sum_{i=1}^n(A_i x_i - b_i) = 0$. 
The variable $\mathbf y$  is introduced to parametrize $\text{range } \mathbf W$.
The cost of this, however, is that after addition of $\by$, the objective $F(\bx, \by)$ is no longer strongly convex. 
Thus, the essential technique to utilize strong convexity is to add augmented-Lagrangian-type penalization term in the objective with proper scaling (see, e.g.~\cref{lem:conv_smooth_c_and_l}).

The following definition is needed to describe convergence rates of decentralized algorithms for problems with coupled constraints.

\begin{lemma}[{\citep[Lemma 2]{yarmoshik2024decentralized}}] \label[lemma]{lem:coupled_spectrum}
	Denote $\bB = \pmat{\bA & \beta \bW}$.
	Setting $\beta^2 = \frac{\lminp(S_A) + \smax^2(\bA)}{\sminp^2(\bW)}$, we get
	\begin{align*}
		\smax^2(\bB)&\leq \smax^2(\bA) + (\smax^2(\bA) + \lminp(S_A))\kappa_W^2, \\
		\sminp^2(\bB)&\geq \frac{\lminp(S_A)}{2}.
	\end{align*}
\end{lemma}

\begin{theorem}[\cite{yarmoshik2024decentralized}, Theorems 1 and 2]\label{thm:coupled_iclr}
    After penalizing \eqref{prob:coupled_constraints_prob} and applying Chebyshev accelerations $\bW\to P_W(\bW)$ and then $\bB\to P_B(\bB^\top \bB)$ and
    applying Algorithm~\ref{alg:apapc}, we obtain a method that requires
    \begin{align*}
        \Nnabla &= O\cbraces{\sqrt{\kappa_f}\log\cbraces{\frac1\eps}} \text{gradient calls}, \\
        \Na &= O\cbraces{\sqrt{\kappa_f} \sqrt{\hat\kappa_A} \log\cbraces{\frac1\eps}} \text{mul. by } \bA \text{ and } \bA^\top, \\
        \Nw &= O\cbraces{\sqrt{\kappa_f} \sqrt{\hat\kappa_A} \sqrt{\kappa_W} \log\cbraces{\frac1\eps}} \text{communications}.
    \end{align*}
    This bound is optimal in a naturally defined class of decentralized first-order algorithms for problems with coupled constraints.
\end{theorem}

\newpage

\section{Proof of the Complexity Bounds for Shared Variable Constraints (\Cref{thm:LN_upper_bound})}\label{app:nonind_local_lower}
\subsection{The upper bound}\label[app]{app:nonind_upper}
As announced in the theorem's statement, we first apply Chebyshev's preconditioning (\Cref{subsec:chebyshev_acc}) to matrix $\bW$ and obtain matrix $\bW' = P_W(\bW)$ with 
$\kappa_{\bW'} = O(1)$.
Then, by \Cref{lem:coupled_spectrum}, choosing $\gamma$ according to its statement, we obtain for $\widetilde{\bB}^\top = \pmat{\widetilde{\bC} & \gamma \bW'}$ 
\begin{equation}
  \kappa_{\widetilde{\bB}} 
  = \frac{\smax^2(\widetilde{\bB})}{\sminp^2(\widetilde{\bB})} 
  = \frac{\smax^2(\widetilde{\bB}^\top)}{\sminp^2(\widetilde{\bB}^\top)} 
  \leq \frac{\smax^2(\widetilde{\bC}^\top) + \left[\smax^2(\widetilde{\bC}^\top) + \lminp(S_{C^\top})\right]\kappa^2_{\bW'}}{\frac12\lminp(S_{C^\top})} = O(\kappa^2_{\bW'}\hkappa_{\widetilde{\bC}}) =  O(\hkappa_{\widetilde{\bC}}).
\end{equation}
Finally, due to \Cref{thm:apapc}, iteration complexity of \Cref{alg:apapc} applied to the problem 
\begin{equation}
  \min_{\mathbf{x} \in (\mathbb{R}^d)^n} F(\mathbf{x}) := \sum_{i=1}^n f_i(x_i)
  \textrm{ s.t. } \widetilde{\bB}' x = \tilde{\bb}',
\end{equation}
is $N = O(\sqrt{\kappa_f}\log \frac{1}{\eps})$, where $\widetilde{\bB}' = P_{\widetilde{B}}(\widetilde{\bB}^\top\widetilde{\bB})$, $\kappa_{\widetilde{\bB}'} = O(1)$ and $\tilde{\bb}' = \frac{P_{\widetilde{B}}(\widetilde{\bB}^\top\widetilde{\bB})}{\widetilde{\bB}^\top \widetilde{\bB}}\widetilde{\bB}^\top \pmat{\tilde{\bc} \\ 0}$.

Each iteration of \Cref{alg:apapc} in this case requires $1$ computation of $\nabla f$,
$O(\deg P_{\widetilde{\bB}}) = O(\sqrt{\kappa_{\widetilde{\bB}}}) = O(\sqrt{\hkappa_{\widetilde{\bC}}})$ multiplications by $\widetilde{\bB}, \widetilde{\bB}^\top$, each requiring $O(1)$ multiplication by $\widetilde{\bC}, \widetilde{\bC}^\top$ and 
$O(\deg P_{W}) = O(\sqrt{\kappa_W})$ multiplications by $\bW$, i.e., communication rounds, which gives the first part of the theorem.

\subsection{The lower bound}
This proof is a modification of the proof of \citep[Theorem 2]{yarmoshik2024decentralized}.
The main difference is in how we split a constraint matrix to obtain a splitting of the Nesterov's bad function 
$h(z) = \frac{1}{2} z^\top \bM z + \alpha \sqn{z} - z_1$
between nodes in \Cref{sec:local_cons_example_mats}.
In the original case of coupled constraints, matrix $\bM$ is split into two summands $\bM = \bE_1^\top \bE_2 + \bE_2^\top \bE_2$. 
Matrices $\bE_1, \bE_2$ are used to form matrices of the coupled constraint.
Here we take a non-symmetric root matrix $\bE$ such that $\bM = \bE^\top \bE$ and split it into matrices $\bE_1, \bE_2$ by rows to form matrices of local constraints.
For completeness, we describe here the whole construction of the lower bound, since its other parts, such as objective functions defined in \Cref{sec:local_cons_example_funcs}, also required changes (more subtle and technical, though) to be applied in this setup.

\subsubsection{Dual problem}
The proof relies on obtaining Nesterov's function as the objective of the \textit{dual} problem.
The primal and dual variables in our construction have similar component structure, what allows to derive the upper bound on accuracy of an approximate solution to the original problem from the explicit expression for the exact solution of the dual problem and an upper bound on the number of nonzero components in the approximate solution.

Let us derive the dual problem. Consider the primal problem with zero right-hand side in the constraints
\begin{equation}
  \begin{aligned}
   \label{prob:lower_primal}
  &\min_{x_1, \ldots, x_n\in\ell_2}~ \sum_{i=1}^n f_i(x_i)
 \\
  \text{s.t. } &\bC_i x_i = 0 \quad \forall i=1,\ldots,n.
  \\
 &x_1 = \ldots = x_n.
  \end{aligned}
\end{equation}

Simplifying the consensus constraint and combining the affine constraints to a single constraint $\tC x = 0$ (e.g., by vertically stacking matrices $\bC_i$), we rewrite the problem as
\begin{equation}
  \begin{aligned}
  &\min_{x \in\ell_2}~ \sum_{i=1}^n f_i(x)
 \\
  \text{s.t. } &\tC x = 0.
  \end{aligned}
\end{equation}

The dual problem has the form
\begin{equation}\label{prob:lower_dual}
\begin{aligned}
	& \max_{z}\min_{x\in\ell_2} \sbraces{\sum_{i=1}^n f_i(x) - \angles{z, \tC x}} 
	= -\min_z F^*(\tC^\top z),
\end{aligned}
\end{equation}
where $F(x) = \sumin f_i(x)$.

\subsubsection{Example  graph}\label[app]{app:example_graph}
The graph construction is the same as in the lower bound for coupled constraints.

We follow the principle of lower bounds construction introduced in \cite{kovalev2021lower} and take the example graph from \cite{scaman2017optimal}. Let the functions held by the nodes be organized into a path graph with $n$ vertices, where $n$ is divisible by $3$. The nodes of graph $\cG = (\cV, \cE)$ are divided into three groups $\cV_1 = \braces{1, \ldots, n/3}, \cV_2 = \braces{n/3 + 1, \ldots, 2n/3}, \cV_3 = \braces{2n/3 + 1, \ldots, n}$ of $n/3$ vertices each.

Now we recall the construction from \cite{scaman2017optimal}. Maximum and minimum eigenvalues of a path graph have form $\lmax(W) = 2\cbraces{1 + \cos\frac{\pi}{n}},~ \lambda_{\min^+}(W) = 2\cbraces{1 - \cos\frac{\pi}{n}}$. Let $\beta_n = \frac{1 + \cos\cbraces{\frac{\pi}{n}}}{1 - \cos\cbraces{\frac{\pi}{n}}}$. Since $\beta_n\overset{n\to\infty}{\rightarrow} +\infty$, there exists $n = 3m\geq 3$ such that $\beta_n\leq \kappa_W < \beta_{n+3}$. For this $n$, introduce edge weights $w_{i, i+1} = 1 - a\mathbb{I}\braces{i=1}$, take the corresponding weighted Laplacian $W_a$ and denote its condition number $\kappa(W_a)$. If $a = 1$, the network is disconnected and therefore $\kappa(W_a) = \infty$. If $a = 0$, we have $\kappa(W_a) = \beta_n$. By continuity of Laplacian spectra we obtain that for some $a \in [0, 1)$ it holds $\kappa(W_a) = \kappa_W$. Note that $\pi/(n+3)\in[0, \pi/3]$, and for $x\in[0, \pi/3]$ we have $1 - \cos x\geq x^2/4$. We have
\begin{align}\label{eq:n_chi_relation}
	\kappa_W\leq \beta_{n+3} = \frac{1 + \cos\frac{\pi}{n+3}}{1 - \cos\frac{\pi}{n+3}}
	\leq \frac{72(n+3)^2}{\pi^2}\leq \frac{288 n^2}{\pi^2}\leq 32n^2~~~\Rightarrow~~~ \sqrt{\kappa_W}\leq 4\sqrt 2n = O(n).
\end{align}

\subsubsection{Example functions}\label{sec:local_cons_example_funcs}
\newcommand{\hLc}{L'_\bC}
\newcommand{\hmuc}{\mu'_\bC}
\newcommand{\shLc}{\sqrt{\hLc}}
\newcommand{\shmuc}{\sqrt{\hmuc}}
We let $e_1 = (1~ 0~ \ldots~ 0)^\top$ denote the first coordinate vector, let $x$ be composed from two variable blocks $x =\pmat{p \\t }$. 
We set functions $f_i$ to be the same for all nodes, and define them as
\begin{align*}
	f_i(p, t) = \frac{\mu_f}{2} \left\|p \right\|^2 + L_f\sqn{t + \frac{\hLc}{\mu_f}e_1}.
\end{align*}
Correspondingly,

\begin{align}\label{eq:conj_fi}
	f_i^*(u, v) = \frac{1}{2\mu_f} \sqn{u} + \frac{1}{L_f} \sqn{v} - \frac{\hLc}{\mu_f}v_1.
\end{align}

Let
\begin{equation}
  \bE = \pmat{
    1 & 0 & 0 & 0 & \cdots \\
    -1 & 1 & 0 & 0 & \cdots \\
    0 & -1 & 1 & 0 & \cdots \\
    0 & 0 & -1 & 1 & \cdots \\
    \vdots & \vdots & \vdots & \vdots & \ddots \\
  }, \quad
  \tC = \pmat{-\shLc\bE^\top  & \shmuc\bI}
\end{equation}

From \eqref{prob:lower_dual} and \eqref{eq:conj_fi}, the dual problem is
\begin{problem}\label{prob:dual_as_nesterov}
\begin{aligned}
  &\min_z \frac{\hLc}{2\mu_f} \sqn{\bE z} + \frac{\hmuc}{L_f} \sqn{z} - \frac{\hLc}{\mu_f}z_1 
  \\=
  &\min_z \frac{\hLc}{2\mu_f} \<z, \bM z> + \frac{\hmuc}{L_f} \sqn{z} - \frac{\hLc}{\mu_f}z_1 
  \\=
  &\min_z \frac{\hLc}{\mu_f} \cbr{
  \frac12\<z, \bM z> + \frac{\hmuc\mu_f}{\hLc L_f} \sqn{z} - z_1 }
  ,
\end{aligned}
\end{problem}
where 
\begin{align*}
	\bM = \bE^\top \bE = \begin{pmatrix}
		~~~2 & -1 & ~~~0 & ~~~0 & 0 & \ldots \\
		-1 & ~~~2 & -1 & ~~~0 & 0 & \ldots \\
		~~~0 & -1 & ~~~2 & -1 & 0 & \ldots \\
		\vdots & \vdots & \vdots & \vdots & \vdots & \ddots
	\end{pmatrix}.
\end{align*}
\Cref{prob:dual_as_nesterov} is exactly the Nesterov's worst problem for smooth strongly convex minimization by first-order methods.
\begin{lemma}[{\citep[Lemma 6]{yarmoshik2024decentralized}}]\label[lemma]{lem:nesterovs_worst_rho}
	The solution of \Cref{prob:dual_as_nesterov} is $z^* = \braces{\rho^k}_{k=1}^\infty$, where
	\begin{align*}
		\rho = \frac{\sqrt{\frac{2}{3} \frac{L_\bC L_f}{\mu_\bC \mu_f} + 1} - 1}{\sqrt{\frac{2}{3} \frac{L_\bC L_f}{\mu_\bC \mu_f} + 1} + 1}.
	\end{align*}
\end{lemma}

\subsubsection{Example matrices}\label{sec:local_cons_example_mats}
We split matrix $\tC$ into two matrices as follows: even rows of $\tC$ go to the first matrix, odd rows are filled with zeros; the second matrix is constructed in the same way from odd rows of $\tC$.
Formally, let $\hLc = \frac{1}{2} L_\bC - \frac{3}{2} \mu_\bC,~ \hmuc = 3\mu_\bC$, where $L_\bC > 0$ and $\mu_\bC > 0$ are any parameters such that $\hat \kappa_{C^\top} = \frac{L_\bC}{\mu_\bC}$,  and introduce
\begin{align*}
	\bC_i = \begin{cases}
    \pmat{\shLc \bE_1^\top & \shmuc \bI_1}, &i\in\cV_1 \\
    \pmat{\mathbf{0} & \qquad\qquad\mathbf{0}\qquad },  &i\in\cV_2 \\
    \pmat{\shLc \bE_2^\top & \shmuc \bI_2}, &i\in\cV_3
	\end{cases},
\end{align*}

\begin{equation}\label{eq:E1}
  \bE_1^\top = \pmat{
    1 & -1 & 0 & 0 & 0 &\cdots \\
    0 & 0 & 0 & 0 & 0 &\cdots \\
    0 & 0 & 1 & -1 & 0 &\cdots \\
    0 & 0 & 0 & 0 & 0 &\cdots \\
    \vdots & \vdots & \vdots & \vdots & \vdots & \ddots \\
  }, \quad
  \bI_1 = \pmat{
    1 & 0 & 0 & 0 & \cdots \\
    0 & 0 & 0 & 0 & \cdots \\
    0 & 0 & 1 & 0 & \cdots \\
    0 & 0 & 0 & 0 & \cdots \\
    \vdots & \vdots & \vdots & \vdots & \ddots \\
  },
\end{equation}
\begin{equation}\label{eq:E2}
  \bE_2^\top = \pmat{
    0 & 0 & 0 & 0 & 0 &\cdots \\
    0 & 1 & -1 & 0 & 0 &\cdots \\
    0 & 0 & 0 & 0 & 0 &\cdots \\
    0 & 0 & 0 & 1 &  -1 &\cdots \\
    \vdots & \vdots & \vdots & \vdots & \vdots & \ddots \\
  }, \quad
  \bI_2 = \pmat{
    0 & 0 & 0 & 0 & \cdots \\
    0 & 1 & 0 & 0 & \cdots \\
    0 & 0 & 0 & 0 & \cdots \\
    0 & 0 & 0 & 1 & \cdots \\
    \vdots & \vdots & \vdots & \vdots & \ddots \\
  }.
\end{equation}

It is clear, that so-defined local constraints are equivalent to $\tC x = 0$.

Let us make sure that this choice of $\bC_i$ indeed guarantees that 
$\frac{\underset{i=1,\ldots,n}{\max}\lmax(C_i^\top C_i)}{\lminp(S_{C^\top})} \leq \hat\kappa_{C^\top}$ %
as required by the statement of the theorem and \Cref{def:mixed_condition_number}.

For the numerator we have
\begin{equation}
\max_i \lmax (\bC_i^\top \bC_i) = \max_i \lmax (\bC_i \bC_i^\top ) = \lmax \cbraces{\hLc \bE_1^\top \bE_1 + \hmuc \bI} = 2\hLc + \hmuc = L_\bC.
\end{equation}

For the denominator direct calculation yields
\begin{equation}
\hspace{-2cm}
\bC_1^\top \bC_1= \pmat{
  1&-1&0&0&\cdots&1&0&0&0&\cdots\\
  -1&1&0&0&\cdots&-1&0&0&0&\cdots\\
  0&0&1&-1&\cdots&0&0&1&0&\cdots\\
  0&0&-1&1&\cdots&0&0&-1&0&\cdots\\
  \vdots &\vdots &\vdots &\vdots&\cdots &\vdots &\vdots &\vdots &\vdots&\cdots \\
  1&-1&0&0&\cdots&1&0&0&0&\cdots\\
  0&0&0&0&\cdots&0&0&0&0&\cdots\\
  0&0&1&-1&\cdots&0&0&1&0&\cdots\\
  0&0&0&0&\cdots&0&0&0&0&\cdots\\
  \vdots &\vdots &\vdots &\vdots &\cdots&\vdots &\vdots &\vdots &\vdots &\ddots\\
}, \quad
\bC_2^\top \bC_2 = \pmat{
0&0&0&0&\cdots&0&0&0&0&\cdots\\
0&1&-1&0&\cdots&0&1&0&0&\cdots\\
0&-1&1&0&\cdots&0&-1&0&0&\cdots\\
0&0&0&1&\cdots&0&0&0&1&\cdots\\
\vdots &\vdots &\vdots &\vdots&\cdots &\vdots &\vdots &\vdots &\vdots &\cdots\\
  0&0&0&0&\cdots&0&0&0&0&\cdots\\
  0&1&-1&0&\cdots&0&1&0&0&\cdots\\
  0&0&0&0&\cdots&0&0&0&0&\cdots\\
  0&0&0&1&\cdots&0&0&0&1&\cdots\\
  \vdots &\vdots &\vdots &\vdots&\cdots &\vdots &\vdots &\vdots &\vdots &\ddots\\
}.
\end{equation}
Using this, we get
\begin{align*}
\lminp \cbraces{\frac{1}{n} \sum_{i=1}^n \bC_i^\top \bC_i} 
 &=
 \frac13\lminp\pmat{
   \hLc (\bE_1 \bE_1^\top + \bE_2 \bE_2^\top) & \sqrt{\hLc\hmuc}(\bE_1 \bI_1 + \bE_2 \bI_2)\\
 \sqrt{\hLc\hmuc}(\bI_1^\top \bE_1 + \bI_2^\top \bE_2)  & \hmuc (\bI_1^\top  \bI_1 + \bI_2^\top  \bI_2)}
 \\&=
 \frac13\lminp\pmat{
   \hLc \bL_{\text{path}} & \sqrt{\hLc\hmuc}\bE\\
\sqrt{\hLc\hmuc}\bE^\top  & \hmuc \bI}
 \\&=
 \frac13\lminp\cbr{\pmat{
 \shLc \bE \\ \shmuc \bI}
\pmat{\shLc\bE^\top & \shmuc \bI}}
 \\&=
 \frac13\lminp\cbr{
\pmat{\shLc\bE^\top & \shmuc \bI}
\pmat{
\shLc \bE \\ \shmuc \bI}}
 \\&=
 \frac13\lminp\cbr{\hLc \bM + \hmuc \bI} = \frac{\hmuc}{3} = \mu_\bC,
\end{align*}
where $\bL_{\text{path}} = \bE\bE^\top$ is the Laplacian of infinite (in one direction) path graph, which differs from $\bM$ only in the first diagonal component $\bL_{\text{path}}[1, 1] = 1$.

\subsubsection{Bounding accuracy}\label{sub:bounding_accuracy}
We consider the class of first-order decentralized algorithms for problems with local affine constraints defined as follows
\begin{definition}[{\citep[Definition 1]{yarmoshik2024decentralized_local}}]\label{def:foda}
Denote $\cM_i(k)$, where $\cM_i(0) = \{x_i^0\}$, as the local memory of the $i$-th node at step $k$. %
The set of allowed actions of a first order decentralized algorithm at step $k$ is restricted to the three options
\begin{enumerate}
  \vspace{-0.3cm}
\item Local computation: $\cM_i(k) = \spn\cbraces{\{x, \nabla f_i(x), \nabla f^*_i(x): x \in \cM_i(k) \}}$;
\item Decentralized communication with immediate neighbours: $\cM_i(k) = \spn\cbraces{\{\cM_j(k): \text{edge}~ (i, j) \in E\}}$.
\item Matrix multiplication: $\cM_i(k) = \spn\cbraces{\{b_i, \bC_i^\top \bC_i x: x \in \cM_i(k)\}}$.
  \vspace{-0.3cm}
\end{enumerate}
After each step $k$, an algorithm must provide a current approximate solution $x^k_i \in \cM_i(k)$ and set $\cM_i(k+1) = \cM_i(k)$.
\end{definition}

Without loss of generality, we can assume $x_i^0 = 0$. Recall that we split $x$ in two variable blocks $p$ and $t$.
From the structure of matrix $\bC_1^\top\bC_1$ it is clear, that a node $i \in \cV_1$ on $k$-th step can only increase the number of nonzero components in $p^k_i$ or $t_i^k$ by one if the index of the last nonzero component in $p^k_i$ or $t^k_i$ is odd.
Similarly, by structure of $\bC_2^\top\bC_2$, nodes from $\cV_2$ can only ``unlock'' next zero component if it is odd.
Thus, to increase the number of nonzero components in $p_i^k$ or $t_i^k$ by two, the information must be transmitted from $\cV_1$ to $\cV_2$ and back (or vice versa), what requires to perform $2 n / 3 = \Omega(\sqrt{\kappa_W})$ decentralized communication rounds and two matrix multiplications.

Due to the strong duality, the solution of problem \eqref{prob:lower_primal} can be obtained from the solution of its dual $\eqref{prob:lower_dual}$ as
\begin{equation*}
x^*(z^*) = \pmat{p^* \\ t^*}(\bC^\top z^*) = 
\pmat{\frac{\shLc}{\mu_f} \bE z^*  
\\
\frac{\sqrt{\hat\mu_\bA}}{2L_f}(z^* - \frac{\hLc}{\mu_f}e_1)
}. 
\end{equation*}
Therefore $t^*$ is just a scaled version of $z^*$ up to the first component.
\Cref{lem:nesterovs_worst_rho} and the standard calculation (see \citep[Appendix C.4]{yarmoshik2024decentralized}) then leads to the following bound on the number $q$ of nonzero components in $t_i^k$ required to reach the accuracy 
$\sqn{x_i^k - x_i^*} \leq \eps$:
\begin{equation}
  q \geq \Omega\cbr{\sqrt{\frac{L_\bC L_f}{\mu_\bC \mu_f}} \log\cbr{\frac{1}{\eps}}},
\end{equation}
what translates to the required number of matrix multiplications
\begin{equation}
  N_\bC \geq \Omega\cbr{\sqrt{\frac{L_\bC L_f}{\mu_\bC \mu_f}} \log\cbr{\frac{1}{\eps}}},
\end{equation}
and decentralized communication rounds
\begin{equation}
  N_\bW \geq \Omega\cbr{\sqrt{\kappa_W}\sqrt{\frac{L_\bC L_f}{\mu_\bC \mu_f}} \log\cbr{\frac{1}{\eps}}}.
\end{equation}

The lower bound on the number of gradient computations is obtained using the same sum-trick as in \cite{yarmoshik2024decentralized}:
to each $f_i(x_i)$ we add an independent copy of Nestrov's worst function $h_i(w_i)$ with appropriate parameters (variables $w_i$ are new independent variables without any coupling between different nodes).

\newpage
\section{Proof of \Cref{lem:mixed_coupled_with_local_spectrum}}\label{app:mixed_coupled_with_local_spectrum}
First, the squared maximum singular value of a block matrix is upper bounded by the sum of  the squared maximum singular values of its blocks, therefore

 \begin{align*}
	    \smax^2\cbr{\bB} &\leq \smax^2\cbr{\bA} + \gamma^2\smax^2\cbr{\bW}  + \beta^2\smax^2\cbr\bC.
	\end{align*}

We set coefficients $\alpha$ and $\beta$ to
\begin{equation}\label{eq:block_scaling}
  \alpha^2=\frac{1}{\mu_W} \begin{cases}
    L_A + \frac14\muproj, &\muproj>0\\
    2L_A, &\muproj=0
  \end{cases},
  \quad
  \beta^2 = \frac{1}{\mu_C} \begin{cases}
    L_S + \frac12\muproj, & \muproj > 0\\
    L_S + 2L_A, &   \muproj = 0
  \end{cases},
\end{equation}
where $L_S = \frac1n \smax^2(\bA')$.

We are going to prove the following bound on minimal positive singular value of $\bB$:
  \begin{equation}
    \sminp^2(\bB) \geq 
      \begin{cases}
        \frac14\muproj, & \muproj > 0,\\
        L_A, & \muproj = 0.
      \end{cases}
  \end{equation}

\textbf{Proof of the lower bound on $\sminp^2(\bB)$}

Since $\sminp^2(\bB) = \sminp^2(\bB^\top)$ we will bound the latter.
image
We have
$\ker^\bot \bB^\top = \image \bB = \image \pmat{\bA \\ \bC} + \pmat{\cL^\bot_m \\ 0}$.
Consider arbitrary $z \in \ker^\bot \bB^\top$. 
Using $\image \bW = \cL_m^\bot$, we can represent $z$ as  $z = \pmat{\bA \\ \bC}\xi + \pmat{\PLb \\ 0}\eta$ for some $\xi, \eta$. 

The key technique of the proof is to decompose $z$ into three orthogonal components $z = \pmat{u \\ 0} + \pmat{v \\0} + \pmat{0 \\ w}$,
where $u = \PLb (\bA \xi + \PLb\eta)$, $v = \PL (\bA \xi + \PLb \eta) = \PL \bA \xi$ and $w = \bC \xi$.

The following relations trivially follow from the definition of the decomposition: 
\begin{subequations}
\begin{align}
  u &\in \cL^\bot_m \label{eq:u} \\
  v &\in \cL_m  \label{eq:v}\\
  u + v &\in \image \bA + \cL^\bot_m \label{eq:uv} \\
  w &\in \image \bC \label{eq:w}\\
  \pmat{v \\ w} &\in \image\pmat{\PL \bA \\ \bC} = \ker^\bot\pmat{ \bA^\top \PL & \bC^\top} \label{eq:vw}
\end{align}
\end{subequations}

Following \citep[Lemma~2]{yarmoshik2024decentralized} and using the relations above we bound $\sqn{\bB^\top z}$ as
\begin{align}
  \sqn{\bB^\top z} \aeq{is due to \eqref{eq:v}} \sqn{\pmat{\bA^\top (u + v) + \bC^\top w \\ \bW u}} 
  \ageq{is due to Young's inequality, \eqref{eq:u} and definitions of $L_A$, $\mu_W$ \eqref{eq:mat_Lmu}}
  -L_A \sqn{u} + \frac12\sqn{\bA^\top v + \bC^\top w} + \mu_W \sqn{u}, \label{eq:mu_B_main}
\end{align} 
where \annotate.

Now consider the second term in the rhs.
By \eqref{eq:v} we have $\bA^\top v = \bA^\top \PL v$.
Then, denoting $\bJ = \pmat{\PL \bA \\ \bC}$ and using \eqref{eq:vw} we obtain 
\begin{equation}
  \sqn{\bA^\top v + \bC^\top w} = \sqn{\bA^\top \PL v + \bC^\top w} \geq \mu_\bJ \sqn{\pmat{v \\ w}}.
\end{equation}

To estimate $\mu_\bJ$, we consider a vector $t \in \ker^\bot \bJ = \image \bA^\top \PL + \image \bC^\top$ and decompose it into orthogonal components as $t = y + q$, 
where $y = \PCb t$ and $q = \PC t$. Then we apply Young's inequality
\begin{equation}
  \sqn{\bJ t} = \sqn{\PL \bA(y +q)} + \sqn{\bC t} \geq -\smax^2(\PL \bA) \sqn{y} +  \frac12 \sqn{\PL \bA q} + \mu_C \sqn{y}.
\end{equation}
Let us bound the first and the second terms in the rhs.
First, $\sqn{\PL \bA y} = \sqn{\frac{1}{n}\one_n\otimes \bA' y} = \frac1n \sqn{\bA' y} \leq \frac1n \smax^2\cbr{\bA'} \sqn{y}$ for any $y$.
Since $\frac 1n \smax^2\cbr{\bA'} = \smax\cbr{\frac1n \one_n\sumin \bA_i \bA_i^\top} = \smax(\bS)$, we denote the coefficient by $L_S$.
Second, by definitions of $q$ and $t$ we have $q = \PC t \in \image \PC \bA^\top \PL  = \image \PC \bA'^\top = \ker^\bot \bA' \PC$.
Therefore, $\sqn{\PL \bA q} = \sqn{\frac1n \one_n \otimes \bA'q} = \sqn{\frac1n \one_n \otimes \bA' \PC q} \geq \frac1n \sminp^2\cbr{\bA' \PC} \sqn{q}= \muproj \sqn{q}$.
Here we allow $\muproj$ to be equal to zero if $\bA' \PC$ is the zero matrix.
Summarizing, we get
\begin{equation}
  \sqn{\bJ t} \geq -L_S \sqn{y} +  \frac12 \muproj \sqn{q} + \mu_C \sqn{y}.
\end{equation}

Now let us utilize the properties of scaling coefficients \eqref{eq:block_scaling}.
If $\muproj = 0$, we have $t = y$ and $\sqn{\bJ t} \geq \cbr{\mu_C - L_S} \sqn{t}$, thus ensuring $\mu_C \geq L_S + 2 L_A$ we get $\mu_J \geq 2L_A$.
Otherwise, if $\muproj > 0$, for $\mu_C \geq L_S + \frac12\muproj$ we obtain $\mu_J = \frac12 \muproj$, which cannot be further increased by scaling $\bC$ in this case.

Plugging this into \eqref{eq:mu_B_main} yields 
\begin{equation}
  \sminp^2(\bB) \geq 
    \begin{cases}
      L_A, & \muproj = 0, \\
      \frac14\muproj, & \muproj > 0,
    \end{cases}
\end{equation}
where we assume $\mu_W \geq L_A + \begin{cases}
  L_A, & \muproj = 0  \\
  \frac14\muproj, & \muproj > 0
\end{cases}$.

\textbf{Bounds for $\kappa_B$}

Finally, to simplify the expression for the condition number of $\bB$ we use the following bounds:
\begin{equation}
\begin{aligned}
  \muproj &= \frac1n \sminp^2(\bA' \PC) 
    \leq \frac1n \smax^2(\bA' \PC)
    \leq \frac1n \smax^2(\bA') \smax^2(\PC)
    \\&=
    \frac1n \smax^2(\bA') 
    \leq \frac1n \sumin \smax^2(\bA_i)
    \leq \smax^2(\bA) = L_A,
\end{aligned}
\end{equation}
and
\begin{equation}
  L_S
    = \frac1n \smax^2(\bA') 
    \leq L_A.
\end{equation}

For $\muproj > 0$ this gives
\begin{equation}
  \kappa_B = \frac{L_B}{\mu_B} \leq\frac{4L_A}{\muproj} + \frac{4L_A + \muproj}{\muproj} \frac{L_W}{\mu_W} + \frac{4L_S + 2\muproj}{\muproj}\frac{L_C}{\mu_C} = O\cbr{ \tkappa_{AC}\kappa_W + \tkappa_{AC} \kappa_C},
\end{equation}
and for $\muproj = 0$
\begin{equation}
  \kappa_B = \frac{L_B}{\mu_B} \leq \frac{L_A}{L_A} + \frac{L_A + L_A}{L_A} \frac{L_W}{\mu_W} + \frac{L_S + 2L_A}{L_A}\frac{L_C}{\mu_C} = O\cbr{\kappa_W + \kappa_C}.
\end{equation}

\newpage
\section{Proof of the Complexity Bounds for Coupled and Local Constraints (\Cref{thm:mixed_upper_bound})}\label{app:mixed_coupled_with_local_lower}
\subsection{The upper bound}\label{sec:mixed_upper}
Since coupled constraints require introducing auxiliary variable $y$ on which the objective function $F(\bx)$ does not depend, and therefore is not strongly convex with respect to the whole set of variables what does not allow to apply \Cref{alg:apapc}. 
Therefore we introduce a regularized/penalized (in the augmented Lagrangian fashion) objective
\begin{equation} \label{eq:coupled_constraints_reg_form}
  G(\bx, \by) = \sumin f_i(x_i) + \frac{r}{2} \sqn{\bA \bx + \alpha \bW \by - \bb}, \quad r = \frac{\mu_f}{2 L_A}.
\end{equation}
The next lemma shows that $G(\bx, \by)$ fixes the strong convexity problem.
\begin{lemma}[essentially {\citep[Lemma 1]{yarmoshik2024decentralized}}]\label[lemma]{lem:conv_smooth_c_and_l}
  $G(\bx, \by)$ is $(\mu_G = \mu_f/ 4)$-strongly convex on $\R^d \times \cY$ and $(L_G = 2L_f\kappa^2_W)$-smooth, where $\cY$ is the subspace of all $\by\in{R^{mn}}$ (recall that $A_i \in \R^{m\times d_i}$) such that $\<\by, \one_{mn}> = 0$.
\end{lemma}
\begin{proof}
  Let $D_G(\bx', \by'; \bx, \by)$ denote the Bregman divergence of $G$:
	\begin{equation}\label{eq:bregman}
		D_G(\bx', \by'; \bx, \by) = G(\bx', \by') - G(\bx, \by) - \<\nabla_x G(\bx, \by), \bx' -\bx> - \<\nabla_y G(\bx, \by), \by' - \by>.
	\end{equation}

	The value of $\mu_G$ can be obtained as follows:
	\begin{align*}
		D_{G}(x',y';x,y)
		 & =
     D_F(x';x) + \frac{r}{2}\sqn{\bA(x'-x) + \alpha \bW (y'-y)}
		\\&\ageq{is due to Young's inequality}
		\frac{\mu_f}{2}\sqn{x'-x} + \frac{r}{4}\sqn{\alpha \bW(y'-y)} - \frac{r}{2}\sqn{\bA(x'-x)}
    \\&\ageq{is due to $y' -y \in \cY$ and  $\ker^\bot \bW = \cY$}
		\frac{\mu_f}{2}\sqn{x'-x} + \frac{r \alpha^2 \mu_\bW}{4}\sqn{y'-y} - \frac{rL_{\bA}}{2}\sqn{x'-x} 
		\\&\geq
		\frac{\mu_f}{4}\sqn{x'-x} + \frac{\mu_f\alpha^2 \mu_\bW}{8 L_{\bA}}\sqn{y'-y},
        \\&\ageq{is because $\alpha^2 \geq \frac{L_A}{\mu_W}$ by its definition in \eqref{eq:block_scaling}}
    \frac{\mu_f}{8} \left\|\pmat{x'-x \\ y' - y}\right\|^2,
	\end{align*}
	where \annotate.

	The value of $L_G$ can be obtained as follows:
	\begin{align*}
		D_{G}(x',y';x,y)
		 & =
		D_F(x';x) + \frac{r}{2}\sqn{\bA(x'-x) + \alpha \bW (y'-y)}
		\\&\aleq{is due to Young's inequality}
		\frac{L_f}{2}\sqn{x'-x} + r\sqn{\alpha \bW (y'-y)} + r\sqn{\bA(x'-x)}
    \\&\leq
		\frac{L_f}{2}\sqn{x'-x} + r \alpha^2 L_\bW \sqn{y'-y} + rL_{\bA}\sqn{x'-x}
		\\&\leq
		\frac{L_f + \mu_f}{2}\sqn{x'-x} + \frac{\mu_f\alpha^2 L_\bW}{2L_{\bA}}\sqn{y'-y},
    \\&\aleq{is because $\alpha^2 \leq 2\frac{L_A}{\mu_W}$ by its definition in \eqref{eq:block_scaling} and $\mu_f \leq L_f$}
    L_f\frac{L_\bW}{\mu_\bW}\left\|\pmat{x'-x \\ y' - y}\right\|^2,
	\end{align*}
	where \annotate.
\end{proof}

Thus we replace $\sumin f_i(x_i)$ with $G(\bx, \by)$ in the decentralized-friendly reformulation of problem~\eqref{prob:coupled_with_local}
\begin{align*}\label{prob:coupled_with_local}
	\min_{\bx, \by}~ &\sumin f_i(x_i)
  \\
  \text{s.t.}~ &\bB \pmat{\bx \\ \by} = \pmat{\bb \\ \beta \bc}, \quad \bB = \pmat{\bA & \alpha\bW\\ \beta\bC & 0}.
\end{align*}
Note, that this does not change the minimizer of the problem.

Next we apply Chebyshev's preconditioning (\Cref{subsec:chebyshev_acc}) and replace matrices $\bW$ and $\bC$ with $\bW' = P_W(\bW)$, $\kappa_{W'} = O(1)$ and $\bC' = P_C(\bC^\top\bC)$, $\kappa_{C'} = O(1)$ correspondingly so that $\bB = \pmat{\bA & \alpha\bW'\\ \beta\bC' & 0}$.
Then, by \Cref{lem:mixed_coupled_with_local_spectrum}, choosing $\alpha$ and $\beta$ as defined in \eqref{eq:block_scaling}, we obtain  $\kappa_B = O(\kappaproj \kappa_{W'} + \kappaproj \kappa_{C'})$ = $O(\kappaproj)$, and by \Cref{lem:conv_smooth_c_and_l} $\kappa_G = \frac{8L_f\kappa^2_{W'}}{\mu_f} = O(\kappa_f)$.

Finally, due to \Cref{thm:apapc}, iteration complexity of \Cref{alg:apapc} applied to the following equivalent reformulation of the problem  \eqref{prob:coupled_with_local}
\begin{equation}
  \min_{\mathbf{x} \in \mathbb{R}^d, \by \in \R^{mn}} G(\mathbf{x}, \by)
  \textrm{ s.t. } \bB' \pmat{\bx \\ \by} = \bb',
\end{equation}
where $\bB' = P_B(\bB^\top\bB)$, $\kappa_{\bB'} = O(1)$ and $\bb' = \frac{P_B(\bB^\top\bB)}{\bB^\top \bB}\bB^\top \pmat{\bb \\ \bc'}$, $\bc' = \beta \frac{P_C(\bC^\top\bC)}{\bC^\top \bC}\bC^\top \bc$ 
is $N = O(\sqrt{\kappa_G}\log \frac{1}{\eps}) =  O(\sqrt{\kappa_f}\log \frac{1}{\eps})$.

Each iteration of \Cref{alg:apapc} in this case requires $1$ computation of $\nabla f$,
$O(\deg P_{\bB}) = O(\sqrt{\kappa_{\bB}}) = O(\sqrt{\kappaproj})$ multiplications by $\bB, \bB^\top$, each requiring $O(1)$ multiplication by $\bA, \bA^\top$, $O(\deg P_C)= O(\sqrt{\kappa_C})$ multiplication by $\bC, \bC^\top$ and 
$O(\deg P_{W}) = O(\sqrt{\kappa_W})$ multiplications by $\bW$, i.e., communication rounds, what gives the first part of the theorem.

Note, that we correctly applied \Cref{lem:conv_smooth_c_and_l} here despite the condition $\by \in \cY$ because when being started from $\by =0$ \Cref{alg:apapc} will keep all iterates $\by^k$ within the subspace $\cY$ due to $\image \bW = \image \bW' = \cY$.

\subsection{The lower bound}
Our worst problem example falls in the case $\muproj > 0$.

We use the proof of the lower bound for coupled constraints from \citep[Theorem~2]{yarmoshik2024decentralized} combined with the idea of the two-level communication graph used in the proof of the lower bound for local constraints \citep[Theorem~2]{yarmoshik2024decentralized_local}.

Let $W\in \R^{n \times n}$ be defined as in \Cref{app:example_graph}.
Let $W_C$ be also a Laplacian matrix with $\frac{\lmax(W_C)}{\lminp(W_C)} = \kappa_C$ constructed as in \Cref{app:example_graph}, and denote the number of vertices in the corresponding path graph as $l$, so that $W_C \in \R^{l\times l}$.

Consider the following objective function
\begin{align}\label{eq:lower_bound_functions_f}
  f_i(p_i, t_i) = \frac{\mu_f}{2} \frac1l \sum_{j=1}^l \sqn{p_{ij} + \frac{\sqrt{L'_A}}{2\mu_f}e_1} + \frac{L_f}{2} \frac1l \sum_{j=1}^l \sqn{t_{ij}},
\end{align}
where $x_i = (p_i, t_i)$, $p_i$ and $t_i$ are $\in \ell_2^l$, and $e_1 = (1~ 0~ \ldots~ 0)^\top$ denote the first coordinate vector.
We immediately note that each $f_i$ is $L_f / l$-smooth and $\mu_f / l$-strongly convex, thus the condition number of $\sumin f_i(x_i)$ is indeed $\kappa_f = \frac{L_f}{\mu_f}$. 

Define $C_i = \pmat{\sqrt{W_C \otimes I_{\ell_2}} & 0 \\ 0 & \sqrt{W_C \otimes I_{\ell_2}}}$ and $b_i = 0$ for all $i \in \{1, \ldots, n\}$.
This gives the desired condition number: 
$\frac{\smax^2(C)}{\sminp^2(C)} = \frac{\lmax(W_C)}{\lminp(W_C)} = \kappa_C$.
The local constraint $C_i \pmat{p_i \\ t_i} = 0$ expresses two consensus constraints: indeed, the first block is equivalent to $(W_C \otimes I_{\ell_2}) p_i = 0$ and, in turn, to $p_{i,1} = \ldots = p_{i,l}$, since $\ker W_C = \{\alpha 1_l~|~\alpha \in \R\}$; same, the second block gives $t_{i,1} = \ldots = t_{i,l}$.

Recall, that, by construction, $n$ and $l$ are divisible by $3$.
We define index sets 
$\cV_1 = \braces{1, \ldots, n/3}, \cV_2 = \braces{n/3 + 1, \ldots, 2n/3}, \cV_3 = \braces{2n/3 + 1, \ldots, n}$,
and, similarly, 
$\cU_1 = \braces{1, \ldots, l/3}, \cU_2 = \braces{l/3 + 1, \ldots, 2l/3}, \cU_3 = \braces{2l/3 + 1, \ldots, l}$.
Let
  \begin{align*}
  &E_1 = \begin{pmatrix}
		1 & 0 & 0 & 0 & 0 & \ldots & \\
		0 & 1 & -1 & 0 & 0 & \ldots & \\
		0 & 0 & 0 & 0 & 0 & \ldots & \\
		0 & 0 & 0 & 1 & -1 & \ldots & \\
		\vdots & \vdots & \vdots & \vdots & \vdots & \ddots \\
	\end{pmatrix},~~~~
	E_2 = \begin{pmatrix}
		1 & -1 & 0 & 0 & 0 & \ldots & \\
		0 & 0 & 0 & 0 & 0 & \ldots & \\
		0 & 0 & 1 & -1 & 0 & \ldots & \\
		0 & 0 & 0 & 0 & 0 & \ldots & \\
		\vdots & \vdots & \vdots & \vdots & \vdots & \ddots & \\
	\end{pmatrix},
  \\
   &A_i = \begin{cases}
    \pmat{\diag(\delta_{\cU_1}) \otimes\sqrt{{L'}_A} E_1^\top & \diag(\delta_{\cU_1}) \otimes\sqrt{{\mu'}_{A}} I_{\ell_2^l}}, & i \in \cV_1 
       \\
       0, &i \in \cV_2
       \\
     \pmat{\diag(\delta_{\cU_3}) \otimes  \sqrt{{L'}_A} E_2^\top &\diag(\delta_{\cU_3}) \otimes   \sqrt{{\mu'}_{A}} I_{\ell_2^l}}, & i \in \cV_3 
   \end{cases} 
   ,
\end{align*}
where $\delta_{\cU_k} \in \R^l$ and 
$[\delta_{\cU_k}]_j = \begin{cases}
    1, &j \in \cU_k
    \\
    0, &\text{otherwise}
\end{cases}$.

Note, that if we presolve the local consensus constraints, we obtain exactly the same problem as in \citep[Appendix C.3]{yarmoshik2024decentralized}. %
Therefore, the solution to the problem considered here is the solution of the problem in \citep[Appendix C]{yarmoshik2024decentralized} copied $l$ times at each vertex.
By \Cref{def:projected_condition_number} 
\begin{equation}\label{eq:tkappaAap}
  \kappaproj = \frac{\underset{i=1,\ldots,n}{\max}\lmax(A_i A_i^\top)}{\frac1n\sminp^2\cbr{\bA' \PC}} 
\end{equation}
To estimate $\underset{i=1,\ldots,n}{\max}\lmax(A_i A_i^\top)$ we rearrange columns of $A_i$ so that it becomes a block-diagonal matrix
\begin{equation}\label{eq:Ai_rear}
   A_i = \begin{cases}
    \diag(\delta_{\cU_1}) \otimes \pmat{\sqrt{{L'}_A} E_1^\top & \sqrt{{\mu'}_{A}} I_{\ell_2^l}}, & i \in \cV_1 
       \\
       0, &i \in \cV_2
       \\
     \diag(\delta_{\cU_3}) \otimes \pmat{ \sqrt{{L'}_A} E_2^\top &  \sqrt{{\mu'}_{A}} I_{\ell_2^l}}, & i \in \cV_3 
   \end{cases}
   ~~=:~~
   \begin{cases}
    \diag(\delta_{\cU_1}) \otimes \bar A_1, & i \in \cV_1 
       \\
       0, &i \in \cV_2
       \\
     \diag(\delta_{\cU_3}) \otimes \bar A_2, & i \in \cV_3 
   \end{cases}
\end{equation}
and by the proof of \citep[Theorem 2]{yarmoshik2024decentralized}, maximal squared singular values of its blocks are upper bounded by $\smax^2(\bar A_k) \leq 2 L'_A + \mu'_A = L_A, ~k\in\{1,2\}$ as needed, if we set 
$L'_A = \frac{1}{2} L_A - \frac{9}{2} \mu_A,~ \mu'_A = 9\muproj$.

For the denominator of \eqref{eq:tkappaAap} we have 
$\ker C_i = \{(p_i, t_i):~ p_{i1} = \ldots = p_{il},~ t_{i1} = \ldots = ~t_{il}\}$,
$\bP_{\ker C_i} =\begin{pmatrix} 1 & 0 \\ 0 & 1 \end{pmatrix} \otimes \cbr{ \frac{1}{l} 1_l 1^\top_l \otimes I_{\ell_2}}$, thus, calculating $\bA'\PC$ and rearranging its columns as in \eqref{eq:Ai_rear} we obtain
\begin{equation}
  \bA'\PC = \begin{pmatrix} 1^\top_{|\cV_1|} \otimes \frac{1}{l} 1_{|\cU_1|}1^\top_{l} \otimes \bar A_1 & 0 & 0
    \\
    0 & 0 & 1^\top_{|\cV_2|} \otimes \frac{1}{l} 1_{|\cU_2|}1^\top_{l} \otimes \bar A_2 
  \end{pmatrix}.
\end{equation}
Consider, for example, the first nonzero diagonal block
\begin{equation}
  \frac1n\sminp^2\cbr{ 1^\top_{|\cV_1|} \otimes \frac{1}{l} 1_{|\cU_1|}1^\top_{l}\otimes \bar A_1}
  = \frac{ |\cV_1| |\cU_1| l}{n l^2} \sminp^2(\bar A_1) 
  = \frac19 \sminp^2(\bar A_1) 
  \geq \frac19 \mu'_A = \muproj.  
\end{equation}
The same holds for the block with $\bar A_2$, thus setting $\muproj = 1$ and $L_A = \kappaproj$ the condition number of $A$ is indeed $\kappaproj$.

With this we verified that the considered problem instance satisfies the assumption on the values of objective function and constraint matrices condition numbers. 
Now let us obtain the lower bounds for the oracle complexities of solving this problem instances by any first-order method.

The solution of the problem has the property that the number of nonzero components $\nu = \max \{k ~:~ \exists i, j:~ [p_{i, j}]_k > 0 \}$ in an approximate solution $x$ is lower bounded by the squared distance $\eps$ between $x$ and the exact solution $x^*$ as $\nu \geq \Omega\cbr{\sqrt{\kappa_A \kappa_f}\log\cbr{\frac1\eps}}$ \citep[Theorem 2]{yarmoshik2024decentralized}. 
Assuming without loss of generality that the starting point for the algorithm is zero, it is clear from the structure of the problem (and, especially, from the structure of $E_1$) that nodes in $\cV_1$  cannot increase the maximum number of nonzero components in any of $x_{i,j}$ by themselves for more than $1$ by performing any operation allowed for a first order method, i.e. primal/dual gradient computation, multiplications by $C_i, C_i^\top$, $A_i, A_i^\top$ and decentralized communication (as well as nodes in $\cV_3$; nodes in $\cV_2$ cannot increase $\nu$ by local operations).
Thus the information is needed to be transferred from $\cV_1$ to $\cV_3$ (or vice versa) in order to increase the number of nonzero components by $O(1)$ what costs $n/3 \geq \Omega\cbr{\sqrt{\kappa_W}}$ communication rounds.
This part of the proof coincides with the case of coupled constraints.

Now, the difference introduced with local constraints is that for nodes in $\cV_1$ multiplication by matrix $E_1$ (which is the only way for them to increase $\nu$) is performed only for $j \in \cU_1$, while $E_2$ applies to $j \in \cU_3$.
Therefore, any algorithm is required to transfer the information between $\cU_1$ and $\cU_3$ to increase $\nu$ by $1$, what costs $l/3 \geq \Omega\cbr{\sqrt{\kappa_C}}$ multiplications by $C_i, C_i^\top$.

Bringing this all together, we obtain
\begin{align}
   N_A &\geq \Omega \cbr{\sqrt{\kappaproj \kappa_f}\log\cbr{\frac1\eps}}
   \\
   N_W &\geq \Omega \cbr{\sqrt{\kappa_W \kappaproj \kappa_f}\log\cbr{\frac1\eps}}
   \\
   N_C &\geq \Omega \cbr{\sqrt{\kappa_C \kappaproj \kappa_f}\log\cbr{\frac1\eps}}
\end{align}

The bound on $N_{\nabla f}$ is obtained by taking the sum of the constructed objectives $f_i(x_i)$ with independent copies of Nestrov's worst functions $h_i(w_i)$ with appropriate parameters on each node (variables $w_i$ are new independent variables without any coupling between different nodes).

\newpage
\section{Proof of the Upper Bounds for Mixed Constraints (Theorem \ref{thm:upper_bound_coupled_plus_non_ident_local_constraints_prob})}
\label{app:upper_bound_coupled_plus_non_ident_local_constraints_prob}

Define 
$\mathbf{K} = \mathrm{diag}(\mathbf{B}, \widetilde{\mathbf{B}})$, where $\bB$ is the matrix of coupled and local constraints  
$\mathbf{B} = \begin{pmatrix} \mathbf{A} & \alpha W \otimes I_m \\ \beta \bC & 0 \end{pmatrix}$, 
$\mathbf{A} = \mathrm{diag}(A_1, \ldots, A_n)$,
and  $\widetilde{\bB}$ is the matrix of consensus and local constraints 
$\widetilde{\mathbf{B}} = \begin{pmatrix} \widetilde \bC \\ \gamma W \otimes I_d \end{pmatrix}$.

Since $\bK$ is block-diagonal, we obtain its preconditioning by separately preconditioning $\bB_1$ as described in \Cref{sec:mixed_upper}, and $\bB_2$ as described in \Cref{app:nonind_upper}.

\subsection{Non-identical local constraints}%

\begin{proof}

  We apply Chebyshev's preconditioning (\Cref{subsec:chebyshev_acc}) and replace matrices $W$ and $\bC$ with $W' = P_W(W)$, $\kappa_{W'} = O(1)$  and $\bC' = P_C(\bC^\top\bC)$, $\kappa_{C'} = O(1)$ respectively, so that $\bB_1 = \pmat{\bA & \alpha W' \otimes I_m \\ \beta \bC' & 0}$  and $\bB_2^\top = \pmat{\widetilde \bC^\top& \gamma W' \otimes I_d}$. 
Then we replace $\bB_1$ with $\bB'_1 = P_{\bB_1}(\bB_1^\top\bB_1)$, $\kappa_{\bB_1'} = O(1)$,
and $\bB_2$ with $\bB'_2 = P_{\bB_2}(\bB_2^\top\bB_2)$, $\kappa_{\bB_2'} = O(1)$.
This preconditioning gives us $\sminp(\bB'_1), \sminp(\bB'_2) \geq 11/15$ 
and $\smax(\bB'_1), \smax(\bB'_2) \leq 19/15$ {\citep[Section 6.3.2]{salim2022optimal}}
thus for $\bK = \diag(\bB'_1, \bB'_2)$ we get $\kappa_K = O(1)$.

To fix the strong convexity issues with coupled constraints (see \Cref{sec:mixed_upper}), we penalize/regularize the objective: $G(\bx, \tilde \bx, \by) = \sumin f_i(x_i, \tilde x_i) + \frac{r}{2}\sqn{\bA \bx + \alpha \bW'\by - \bb}$ .
By {\citep[Lemma 1]{yarmoshik2024decentralized}} for $r = \frac{\mu_f}{2 L_A}$ we have $\kappa_G = O(\kappa_f)$.

Now, due to \Cref{thm:apapc}, iteration complexity of \Cref{alg:apapc} applied to the following equivalent reformulation of problem \eqref{prob:opt_mixed_constraints_intro}
\begin{equation}
  \min_{\mathbf{x}, \tilde \bx, \by} G(\mathbf{x}, \tilde \bx, \by)
  \textrm{ s.t. } \bK \pmat{\bx \\ \by \\ \tilde \bx} = \pmat{\bb' \\ \bc'},
\end{equation}
where $\bb' = \frac{P_{\bB_1}(\bB_1^\top\bB_1)}{\bB_1^\top \bB_1}\bB_1^\top \pmat{\bb \\ \bc}$, $\bc' = \frac{P_{\bB_2}(\bB_2^\top\bB_2)}{\bB_2^\top \bB_2}\bB_2^\top \pmat{\tilde \bc\\ 0}$
is $N = O(\sqrt{\kappa_G}\log \frac{1}{\eps}) =  O(\sqrt{\kappa_f}\log \frac{1}{\eps})$.
Each iteration of \Cref{alg:apapc} in this case requires 
\begin{itemize}
  \item $1$ computation of $\nabla f$;
  \item $O(1)$ multiplications by $\bB'_1, \bB'^\top_1$ equivalent to $O(\deg P_{\bB_1}) = O(\sqrt{\kappa_{\bB_1}}) = O(\sqrt{\tkappa_{AC}})$ multiplications by $\bB_1, \bB_1^\top$, each requiring $O(1)$ multiplications by $\bA, \bA^\top$,  $O(\deg P_W)= O(\sqrt{\kappa_W})$ multiplications by $\bW$ (communications) and $O(\deg P_C) = O(\sqrt{\kappa_C})$ multiplications by $\bC$, $\bC^\top$; 
  \item $O(1)$ multiplications by $\bB_2, \bB'^\top_2$ equivalent to $O(\deg P_{\bB_2}) = O(\sqrt{\hkappa_{\widetilde \bC^\top}})$ multiplications by $ \bB_2, \bB_2 ^\top$ each requiring $O(1)$ multiplications by $\widetilde \bC, \widetilde \bC^\top$ and $O(\deg P_W)= O(\sqrt{\kappa_W})$ communications. 
\end{itemize}
Counting total number of matrix multiplications concludes the proof.
\end{proof}

\subsection{Identical local constraints}%
In this case we have $\widetilde \bC = I_n \otimes \tilde C$.
\begin{proof}
  The difference is in the application of Chebyshev's preconditioning: instead of taking polynomial of $\bB_2$ we simply replace matrix $\widetilde \bC$ with $\widetilde \bC' = P_C(\bC)$, $\kappa_{\widetilde C'} = O(1)$, what by \Cref{lem:local_spectrum} with $\gamma=1$ also gives that $\bB_2^\top = \pmat{I_n\otimes C'^\top& \gamma W' \otimes I_d}$ has $\sminp(\bB_2),\smax(\bB_2)$ both bounded as $\Theta(1)$. 
$\bB_1$, $W$ and objective function are treated the same way as in the case of non-identical constraints.

Each iteration of \Cref{alg:apapc} in this case requires 
\begin{itemize}
  \item (same as in the non-identical case) $1$ computation of $\nabla f$;
  \item (same as in the non-identical case) $O(1)$ multiplications by $\bB'_1, \bB'^\top_1$ equivalent to $O(\deg P_{\bB_1}) = O(\sqrt{\kappa_{\bB_1}}) = O(\sqrt{\tkappa_{A}})$ multiplications by $\bB_1, \bB_1^\top$, each requiring $O(1)$ multiplications by $\bA, \bA^\top$,  $O(\deg P_W)= O(\sqrt{\kappa_W})$ multiplications by $\bW$ (communications) and $O(\deg P_C) = O(\sqrt{\kappa_C})$ multiplications by $\bC$, $\bC^\top$; 
  \item (reduced complexity by $W$) $O(1)$ multiplications by $\bB_2, \bB'^\top_2$ each requiring $O(\deg P_{\widetilde \bC}) = O(\sqrt{\widetilde \kappa_{\bC}})$ multiplications by $\bC, \bC^\top$ and $O(\deg P_W)= O(\sqrt{\kappa_W})$  communications. 
\end{itemize}
\end{proof}

For the case when $\bC = 0$ we also give the following useful lemma, which allows to precisely estimate the condition number of matrix $\bB$ without applying Chebyshev's preconditioning, but only by scaling matrices with scalars.

\begin{lemma}
    Let $\widetilde C_1 = \ldots = \widetilde C_n = \widetilde C$. Denote $\bB = \diag(\alpha\bB_1, \bB_2)$, where $\bB_1 = \pmat{\bA & \beta W \otimes I_m}$, $\bA = \diag\cbr{A_1, \ldots, A_n}$ and $\bB_2^\top = (I_n\otimes \widetilde C^\top~~ \gamma W \otimes I_d)$.
    Also recall definitions of $S_A$ and $\hkappa_A$ from Definition \ref{def:mixed_condition_number}.
    Setting $\alpha^2 = \frac{2\sminp^2(\widetilde C)}{\lambda_{\min^+}(S_A)}$, $\beta^2 = \frac{\lambda_{\min^+}(S_A) + \smax^2(\bA)}{\sminp^2(W)}$ and $\gamma^2 = \frac{\sminp^2(\widetilde C)}{\sminp^2(W)}$, we obtain
    \begin{align*}
    	\smax^2(\bB) &\leq O\cbraces{\smax^2(\widetilde C) + \sminp^2(\widetilde C)\cdot \kappa_W^2 \hkappa_A}, \\
    	\sminp^2(\bB) &\geq \sminp^2(\widetilde C),
    \end{align*}
    thus $\kappa_{B} = O(\kappa_{\widetilde C} + \kappa^2_W \hkappa_A)$.
\end{lemma}

\begin{proof}
	Let us separately estimate the spectrum of $\bB_1$ and $\bB_2$. Set $\beta^2 = \frac{\lambda_{\min^+}(S_A) + \smax^2(\bA)}{\sminp^2(W)}$. According to Lemma \ref{lem:coupled_spectrum}, we have
	\begin{align*}
		\smax^2(\bB_1) &\leq \smax^2(\bA) + (\smax^2(\bA) + \lambda_{\min^+}(S_A)) \kappa_W^2, \\
		\sminp^2(\bB_1) &\geq \frac{\lambda_{\min^+}(S_A)}{2}.
	\end{align*}
	By Lemma \ref{lem:local_spectrum} we obtain
	\begin{align*}
		\smax^2(\bB_2) &= \smax^2(\widetilde C) + \gamma^2 \smax^2(W), \\
		\sminp^2(\bB_2) &= \min(\sminp^2(\widetilde C), \gamma^2\sminp^2(W)).
	\end{align*}
	Recalling the values of $\alpha^2$ and $\gamma^2$ from the statement of the lemma, we obtain the estimate on $\smax^2(\bB)$:
	\begin{align*}
		\smax^2(\bB)&\leq \max(\alpha^2\smax^2(\bB_1), \smax^2(\bB_2)) 
    \\ &=\max \left( 2\sminp^2(\widetilde C)\cbraces{\hat\kappa_A + (\hat\kappa_A + 1)\kappa_W^2},
    \smax^2(\widetilde C) + \sminp^2(\widetilde C)\kappa_W^2,\right) \\
		&= O\left(\smax^2(\widetilde C) + \sminp^2(\widetilde C) \kappa_W^2 \hat\kappa_A\right).
	\end{align*}
	Analogously we get a bound on $\sminp^2(\bB)$:
	\begin{align*}
		\sminp^2(\bB)&\geq \min(\alpha^2 \sminp^2(\bB_1), \sminp^2(\bB_2)) \\
		&= \min\cbraces{\frac{\alpha^2 \lambda_{\min^+}(S_A)}{2}, \sminp^2(\widetilde C), \gamma^2 \sminp^2(W)} \\
		&=\sminp^2(\widetilde C).
	\end{align*}
\end{proof}

\newpage
\section{Missing Proofs From \Cref{sec:smooth_convex_case}}\label{app:smooth_convex_results}

\begin{table}[H]
	\begin{center}
		\begin{small}
			\begin{sc}
				\begin{tabular}{ccc}
					\toprule
					Prob. & Oracle & Compl. \\
					\midrule
					\multirow{5}{*}{\shortstack{Coupled \\ constr. \\ \eqref{prob:coupled_constraints_prob}}} 
					& Grad. & $\sqrt{\frac{L_f R^2}{\varepsilon}}$ \\
					& Mat. & $\sqrt{\hat\kappa_{A}}\sqrt{\frac{L_f R^2}{\varepsilon}}$ \\
					& Comm. & $\sqrt{\kappa_{W}}\sqrt{\hat\kappa_{A}}\sqrt{\frac{L_f R^2}{\varepsilon}}$ \\
					& Paper & This paper, Th.~\ref{thm:upper_bounds_mixed_smooth_nonstrongly_convex} \\
					\midrule
                    \multirow{4}{*}{\shortstack{Shared \\ var.\\constr. \\ \eqref{prob:local_non_ident_constraints_on_shared_var_prob}}}
					& Grad. & $\sqrt{\frac{L_f R^2}{\varepsilon}}$ \\
					& Mat. $\widetilde C$& $ \sqrt{\hat{\kappa}_{\widetilde{C}^\top}} \sqrt{\frac{L_f R^2}{\varepsilon}}$ \\
					& Comm. & $\sqrt{\hat\kappa_{\widetilde{C}^\top}} \sqrt{\kappa_W}  \sqrt{\frac{L_f R^2}{\varepsilon}}$ \\
					& Paper & This paper, Th.~\ref{thm:upper_bounds_mixed_smooth_nonstrongly_convex} \\
                    \midrule
					\multirow{5}{*}{\shortstack{Local \\ var. \\constr. \\ \eqref{prob:coupled_with_local}}}
					& Grad. &  $\sqrt{\frac{L_f R^2}{\varepsilon}}$ \\
					& Mat. $A$ & $\sqrt{\kappaproj}\sqrt{\frac{L_f R^2}{\varepsilon}}  $ \\
					& Mat. $C$ & $\sqrt{\kappaproj} \sqrt{\kappa_{C}}  \sqrt{\frac{L_f R^2}{\varepsilon}}$ \\
					& Comm. & $\sqrt{\kappaproj} \sqrt{\kappa_W}  \sqrt{\frac{L_f R^2}{\varepsilon}}$ \\
					& Paper & This paper, Th.~\ref{thm:upper_bounds_mixed_smooth_nonstrongly_convex} \\
                    \midrule
                    \multirow{6}{*}{\shortstack{General \\ Mixed \\ constr. \\ \eqref{prob:opt_mixed_constraints_intro}}}
					& Grad. & $\sqrt{\frac{L_f R^2}{\varepsilon}}$ \\
					& Mat. $A$ & $\sqrt{\kappaproj}\sqrt{\frac{L_f R^2}{\varepsilon}}  $  \\
					& Mat. $C$ & $\sqrt{\kappaproj} \sqrt{\kappa_{C}}  \sqrt{\frac{L_f R^2}{\varepsilon}}$ \\
                    & Mat. $\widetilde{C}$ &  $ \sqrt{\hat{\kappa}_{\widetilde{C}^\top}} \sqrt{\frac{L_f R^2}{\varepsilon}}$  \\
					& Comm. & $\left(\sqrt{\kappaproj}+ \sqrt{\hat\kappa_{\widetilde{C}^\top}}\right) \sqrt{\kappa_W} \sqrt{\frac{L_f R^2}{\varepsilon}} $ \\
					& Paper & This paper, Th.~\ref{thm:upper_bounds_mixed_smooth_nonstrongly_convex} \\
					\bottomrule
				\end{tabular}
			\end{sc}
		\end{small}
	\end{center}
	\caption{Convergence rates for decentralized smooth convex optimization with constraints. Constants $ \hat{\kappa}_A, \kappaproj, \hat\kappa_{C^\top}, \kappa_{C}$ are defined similarly to Table \ref{table:smooth_str_convex_optimization_affine_constraints}. The term $\log\left(\frac{1}{\varepsilon}\right)$ is omitted.}
	\label{table:smooth_conv_optimization_constraints}
\end{table}

\subsection{Proof of \Cref{thm:upper_bounds_mixed_smooth_nonstrongly_convex}}
From \Cref{thm:smooth_objectives_apapc_convergence}, we know that to achieve $\varepsilon$-solution, Algorithm~\ref{alg:apapc} requires
 $N_{\nabla f} = \mathcal{O}\left(\sqrt{\frac{LR^2}{\varepsilon}} \log \left( \frac{1}{\varepsilon} \right)\right)$ gradient computations and $N_\bB = \mathcal{O}\left(\sqrt{\frac{LR^2}{\varepsilon}} \sqrt{\kappa_\bB} \log \left( \frac{1}{\varepsilon} \right)\right)$ matrix multiplications with $\bB$ and $\bB^\top$.

Follow the Lemma~\ref{lem:mixed_coupled_with_local_spectrum} with $\alpha$, $\beta$ are chosen as in \eqref{eq:block_scaling}, we obtain
\begin{equation*}
      \kappa_\bB = O\cbr{\kappaproj\kappa_W+ \kappaproj \kappa_C}.
\end{equation*}
 The Chebyshev preconditioning for $\bB$ requires $O(\sqrt{\kappaproj})$ multiplications by $\bA$, $\bA^\top$, $O(\sqrt{\kappaproj} \sqrt{\kappa_C} )$ multiplications by $\bC$, $\bC^\top$, and $O(\sqrt{\kappaproj} \sqrt{\kappa_W} )$ communication rounds.
 
Applying Lemma~\ref{lem:coupled_spectrum} to the matrix $\widetilde{\bB}^\top = (\widetilde{\bC}^\top\;\; \gamma \bW)$ with $\gamma^2 = \frac{\lminp(\bS_{\widetilde{C}^\top}) + \smax^2(\widetilde{\bC}^\top)}{\sminp^2(\bW)}$, we obtain
\[\kappa_{\widetilde{\bB}} = \kappa_{\widetilde{\bB}^\top} \leq 2 \hat{\kappa}_{\widetilde{C}^\top} + 2 \left(\hat{\kappa}_{\widetilde{C}^\top} + 1\right) \kappa_W^2.\]

 Then when we apply Chebyshev acceleration for matrix $\widetilde{\bB}$, it requires to perform $O\left(\sqrt{\hat{\kappa}_{\widetilde{C}^\top}}\right)$ multiplications by $\widetilde{\bC}, \widetilde{\bC}^\top$ and $O\left(\sqrt{\kappa_W} \sqrt{\hat{\kappa}_{\widetilde{C}^\top}}\right)$ multiplications by $\bW$. 

 Therefore, in total, when the Chebyshev preconditioning for $\bK$ requires $O(\sqrt{\kappaproj})$ multiplications by $\bA$, $\bA^\top$, $O(\sqrt{\kappaproj} \sqrt{\kappa_C} )$ multiplications by $\bC$, $\bC^\top$, $O\left(\sqrt{\hat{\kappa}_{\widetilde{C}^\top}}\right)$ multiplications by $\widetilde{\bC}, \widetilde{\bC}^\top$ and $O\left((\sqrt{\kappaproj} + \sqrt{\hat{\kappa}_{\widetilde{C}^\top}})\sqrt{\kappa_W} \right)$ communication rounds. From this we derived the results in the Table~\ref{table:smooth_conv_optimization_constraints}.

 When local matrices $\widetilde{C}_i$ are equal, for the block $\widetilde{\bB}$, preconditioning requires $O(\sqrt{\kappa_C})$ multiplications by $\bC$, $\bC^\top$ and $O(\sqrt{\kappaproj}\sqrt{\kappa_W})$ communication rounds. In that case, the complexities $N_C$ and $N_W$ change to
  \begin{align*}
       N_{\widetilde C} &= O\cbraces{\sqrt{\kappa_{\widetilde C}} \sqrt{\frac{L_f R^2}{\eps}}  \log\cbraces{\frac1\eps}} \\
         N_W &= O\cbraces{ \sqrt{\kappaproj} \sqrt{\kappa_W} \sqrt{\frac{L_f R^2}{\eps}}  \log\cbraces{\frac1\eps}}.
     \end{align*}

\newpage

\section{Missing Proofs From \Cref{sec:nonsmooth_convex_case}}\label{app:nonsmooth_convex_results}

\begin{table}[H]
	\vskip 0.15in
	\begin{center}
		\begin{small}
			\begin{sc}
				\begin{tabular}{ccc}
					\toprule
					Prob. & Oracle & Compl. \\
					\midrule
					\multirow{5}{*}{\shortstack{Coupled \\ constr. \\ \eqref{prob:coupled_constraints_prob}}} 
					& Grad. & $\frac{M_f^2 R^2}{\eps^2}$ \\
					& Mat. & $\sqrt{\hkappa_A}\frac{M_f R}{\eps}$ \\
					& Comm. & $\sqrt{\kappa_W}\sqrt{\hkappa_A}\frac{M_f R}{\eps}$ \\
					& Paper & This paper, Th.~\ref{thm:nonsmooth_convex_upper_bounds}\\
					\midrule
                    \multirow{4}{*}{\shortstack{Shared \\ var.\\ constr. \\ \eqref{prob:shared_var_constraints}}}
					& Grad. & $\frac{M_f^2 R^2}{\eps^2}$ \\
					& Mat. $\widetilde{C}$ & $\sqrt{\hkappa_{\widetilde{C}^\top}}\frac{M_f R}{\eps}$ \\
					& Comm. & $\sqrt{\hat\kappa_{\widetilde{C}^\top}} \sqrt{\kappa_W}\frac{M_f R}{\eps}$ \\
					& Paper & This paper, Th.~\ref{thm:nonsmooth_convex_upper_bounds} \\
					\midrule
					\multirow{5}{*}{\shortstack{Local\\ var. \\ constr. \\ \eqref{prob:coupled_with_local}}}
					& Grad. & $\frac{M_f^2 R^2}{\eps^2}$ \\
					& Mat. $A$ & $ \sqrt{\kappaproj} \frac{M_f R}{\eps}$ \\
					& Mat. $C$ & $\sqrt{\kappaproj} \sqrt{\kappa_{C}}  \frac{M_f R}{\eps}$ \\
					& Comm. & $\sqrt{\kappaproj} \sqrt{\kappa_W} \frac{M_f R}{\eps}$ \\
					& Paper & This paper, Th.~\ref{thm:nonsmooth_convex_upper_bounds} \\
                    \midrule
                    \multirow{6}{*}{\shortstack{Mixed \\ constr. \\ \eqref{prob:opt_mixed_constraints_intro}}}
					& Grad. & $\frac{M_f^2 R^2}{\eps^2}$ \\
					& Mat. $A$ & $ \sqrt{\kappaproj} \frac{M_f R}{\eps}$ \\
					& Mat. $C$ & $\sqrt{\kappaproj} \sqrt{\kappa_{C}}  \frac{M_f R}{\eps}$ \\
                    & Mat. $\widetilde{C}$ & $\sqrt{\hkappa_{\widetilde{C}^\top}}\frac{M_f R}{\eps}$ \\
					& Comm. & $\left(\sqrt{\kappaproj} + \sqrt{\hat\kappa_{\widetilde{C}^\top}}\right) \sqrt{\kappa_W} \frac{M_f R}{\eps}$ \\
					& Paper & This paper, Th.~\ref{thm:nonsmooth_convex_upper_bounds} \\
					\bottomrule
				\end{tabular}
			\end{sc}
		\end{small}
	\end{center}
	\caption{Convergence rates for decentralized nonsmooth and (non-strongly) convex case.}
	\label{table:nonsmooth_conv_optimization_constraints}
\end{table}

\subsection{Proof of \Cref{thm:nonsmooth_convex_upper_bounds}}

Follow the \Cref{thm:upper_bound_nonsmooth_affine_constraints}, the penalized reformulation of problem \eqref{eq:mixed_constraints_reformulation} can be solved by gradient sliding with $N_\bK = O\cbr{\frac{M_fR}{\eps} \sqrt{\kappa_\bK}}$ multiplications by $\bK, \bK^\top$ and $N_{\nabla F} = O\cbr{\frac{M_f^2R^2}{\eps^2} + N_\bK}$ calls of subgradient of $F$.

As shown in Appendix~\ref{app:smooth_convex_results},  the Chebyshev preconditioning for $\bK$ requires $O(\sqrt{\kappaproj})$ multiplications by $\bA$, $\bA^\top$, $O(\sqrt{\kappaproj} \sqrt{\kappa_C} )$ multiplications by $\bC$, $\bC^\top$, $O\left(\sqrt{\hat{\kappa}_{\widetilde{C}^\top}}\right)$ multiplications by $\widetilde{\bC}, \widetilde{\bC}^\top$ and $O\left((\sqrt{\kappaproj} + \sqrt{\hat{\kappa}_{\widetilde{C}^\top}})\sqrt{\kappa_W} \right)$ communication rounds. From this we derived the results in the Table~\ref{table:nonsmooth_conv_optimization_constraints}.

\newpage
\section{Missing Proofs From \Cref{sec:nonsmooth_strongly_convex_case}} \label{app:nonsmooth_strongly_convex_results}

\begin{table}[H]
	\vskip 0.15in
	\begin{center}
		\begin{small}
			\begin{sc}
				\begin{tabular}{ccc}
					\toprule
					Prob. & Oracle & Compl. \\
					\midrule
					\multirow{5}{*}{\shortstack{Coupled \\ constr. \\ \eqref{prob:coupled_constraints_prob}}} 
					& Grad. & $\frac{M_f^2}{\mu_f\eps}$ \\
					& Mat. $A$ & $\sqrt{\hkappa_A}\frac{M_f}{\sqrt{\mu_f\eps}}$ \\
					& Comm. & $\sqrt{\kappa_W}\sqrt{\hkappa_A}\frac{M_f}{\sqrt{\mu_f\eps}}$ \\
					& Paper & This paper, Th.~\ref{thm:nonsmooth_strongly_convex_upper_bounds}  \\
					\midrule
                    \multirow{4}{*}{\shortstack{Shared \\ var. \\ constr. \\ \eqref{prob:shared_var_constraints}}}
					& Grad. & $\frac{M_f^2}{\mu_f\eps}$ \\
					& Mat. $\widetilde C$ & $\sqrt{\hkappa_{\widetilde{C}^\top}}\frac{M_f}{\sqrt{\mu_f\eps}}$ \\
					& Comm. & $\sqrt{\hat\kappa_{\widetilde{C}^\top}}\sqrt{\kappa_W}\frac{M_f}{\sqrt{\mu_f\eps}}$ \\
					& Paper & This paper, Th.~\ref{thm:nonsmooth_strongly_convex_upper_bounds} \\
					\midrule
					\multirow{5}{*}{\shortstack{Local \\ var. \\ constr. \\ \eqref{prob:coupled_with_local}}}
					& Grad. & $\frac{M_f^2}{\mu_f\eps}$ \\
					& Mat. $A$ & $\sqrt{\kappaproj} \frac{M_f}{\sqrt{\mu_f\eps}}$ \\
					& Mat. $C$ & $\sqrt{\kappaproj} \sqrt{\kappa_{C}}  \frac{M_f}{\sqrt{\mu_f\eps}}$ \\
					& Comm. & $\sqrt{\kappaproj} \sqrt{\kappa_W} \frac{M_f}{\sqrt{\mu_f\eps}}$ \\
					& Paper & This paper, Th.~\ref{thm:nonsmooth_strongly_convex_upper_bounds} \\
                    \midrule
                    \multirow{6}{*}{\shortstack{ Mixed \\ constr. \\ \eqref{prob:opt_mixed_constraints_intro}}}
					& Grad. & $\frac{M_f^2}{\mu_f\eps}$ \\
					& Mat. $A$ & $\sqrt{\hat\kappa_A} \frac{M_f}{\sqrt{\mu_f\eps}}$ \\
					& Mat. $C$ & $\sqrt{\kappaproj} \sqrt{\kappa_{C}}  \frac{M_f}{\sqrt{\mu_f\eps}}$ \\
                    & Mat. $\widetilde C$ & $\sqrt{\hkappa_{\widetilde{C}^\top}}\frac{M_f}{\sqrt{\mu_f\eps}}$ \\
					& Comm. & $\left(\sqrt{\kappaproj} + \sqrt{\hat\kappa_{\widetilde{C}^\top}}\right) \sqrt{\kappa_W} \frac{M_f}{\sqrt{\mu_f\eps}}$ \\
					& Paper & This paper, Th.~\ref{thm:nonsmooth_strongly_convex_upper_bounds} \\
					\bottomrule
				\end{tabular}
			\end{sc}
		\end{small}
	\end{center}
	\caption{Convergence rates for decentralized nonsmooth and strongly convex optimization with different types of affine constraints.}
	\label{table:nonsmooth_strconv_optimization_constraints}
\end{table}

In the strongly convex and non-smooth setting, we consider the problem~\eqref{eq:coupled_constraints_org} on a bounded set $\mathcal{X} \times \widetilde{\mathcal{X}}$, where $\mathcal{X} = X_1 \times \dots \times X_n$, otherwise Assumption~\ref{assum:bounded_gradient} and Assumption~\ref{assum:strongly_convex} with $\mu > 0$ cannot be held simultaneously. 
If a first-order algorithm use $\by^0 = \mathbf 0$ as a starting point, then its iterates $\by^k$ will belong to the linear subspace $\mathcal{L}_m^\perp$. Hence we can restrain the search space on this subspace and rewrite the penalized formulation of the problem~\eqref{eq:mixed_constraints_reformulation} as follows

\begin{equation} \label{eq:mixed_constraints_nonsmooth_penalized}
    \min_{\bx\in \mathcal{X},\, \by \in \mathcal{L}_m^\perp,\,
    \tilde{\bx} \in \widetilde{\mathcal{X}}
    } G(\bx, \by, \tilde{\bx})= F(\bx, \tilde{\bx}) + \frac{r^2}{\varepsilon} \sqn{\bK \pmat{\bx \\ \by\\ \tilde{\bx}} - \bv}. %
\end{equation}

\subsection{Proof of Lemma~\ref{lem:strong_conv_G_penal}} \label{app:proof_strong_conv_G_penal}

    Suppose that Assumption~\ref{ass:nonsmooth_and_str_convex_of_f_i} holds with $\mu_f \geq 0 $. Let $\alpha$ and $\varepsilon$ satisfy following conditions:
    \begin{equation} \label{eq:strcv_non_smooth_coupled_constraints_constants_condition}
        \alpha^2 = \frac{\mu_{\bA} + L_{\bA}}{\mu_{\bW}}, \quad \varepsilon \leq \frac{4 r^2 \mu_\bA}{\mu_f}.
    \end{equation}

Let \begin{equation} \label{eq:delta_def}
    \delta = \frac{1}{1+\frac{\mu_f \varepsilon}{4r^2 L_\bA}} \in (0, 1).
\end{equation}
For any $\bz = \col(\bx, \by, \tilde{\bx})$ and $\bz^\prime = \col(\bx^\prime, \by^\prime, \tilde{\bx}^\prime)$, where $\bx, \bx^\prime \in \mathcal{X}$, $\by, \by^\prime \in \mathcal{L}_m^\perp$ and $\tilde{\bx}, \tilde{\bx}^\prime \in \widetilde{\mathcal{X}}$, we have
\begin{align*}
		D_{G}(\bx^\prime,\by^\prime, \tilde{\bx}^\prime;\bx,\by, \tilde{\bx})
		 & = 
     D_F(\bx^\prime, \tilde{\bx}^\prime;\bx, \tilde{\bx}) + \frac{r^2}{\varepsilon}\normsq{\bK(\tilde{\bz}^\prime - \bz)}\\
     &\ageq{is due to Assumption~\ref{assum:strongly_convex} and definition of $\bK$} \frac{\mu_f}{2}\normsq{\bx^\prime-\bx} +  \frac{\mu_f}{2}\normsq{\tilde{\bx}^\prime-\tilde{\bx}} + \frac{r^2}{\varepsilon}\normsq{\bA(\bx^\prime-\bx) + \gamma \bW (\by^\prime-\by)}\
		\\&\ageq{is due to Young's inequality}
		\frac{\mu_f}{2}\normsq{\bx'-\bx} +  \frac{\mu_f}{2}\normsq{\tilde{\bx}^\prime-\tilde{\bx}}  + \frac{r^2(1-\delta)} {\varepsilon}\normsq{\gamma \bW(\by'-\by)} - \frac{r^2}{\varepsilon} \left( \frac{1}{\delta} - 1\right)\normsq{\bA(\bx'-\bx)}
    \\&\ageq{is due to $y' -y \in \mathcal{L}_m^\perp$}
		\frac{\mu_f}{2}\normsq{\bx'-\bx} +  \frac{\mu_f}{2}\normsq{\tilde{\bx}^\prime-\tilde{\bx}}  + \frac{r^2 (1 - \delta) \gamma^2 \mu_\bW}{\varepsilon}\normsq{\by'-\by} - \frac{r^2 L_{\bA}}{\varepsilon} \left( \frac{1}{\delta} - 1\right) \normsq{\bx'-\bx} 
		\\&\aeq{is due substitution of $\delta$ in \eqref{eq:delta_def} and $\alpha^2$ in \eqref{eq:strcv_non_smooth_coupled_constraints_constants_condition}}
		\frac{\mu_f}{4}\normsq{\bx'-\bx} +  \frac{\mu_f}{2}\normsq{\tilde{\bx}^\prime-\tilde{\bx}}  + \frac{r^2 \mu_f (\mu_\bA + L_\bA)}{4 r^2 L_\bA + \mu_f \varepsilon}\normsq{\by'-\by}
        \\&\ageq{is due to the upper bound on $\varepsilon$ in \eqref{eq:strcv_non_smooth_coupled_constraints_constants_condition}}
    \frac{\mu_f}{4}\normsq{\bx'-\bx} +  \frac{\mu_f}{2}\normsq{\tilde{\bx}^\prime-\tilde{\bx}}  + \frac{r^2 \mu_f (\mu_\bA + L_\bA)}{4 r^2 L_\bA + 4r^2 \mu_\bA}\normsq{\by'-\by}
    \\&=\frac{\mu_f}{4} \left\|\pmat{\bx'-\bx \\ \by' - \by}\right\|^2 +  \frac{\mu_f}{2}\normsq{\tilde{\bx}^\prime-\tilde{\bx}}\\
    &> \frac{\mu_f}{4}\normsq{\bz^\prime-\bz}, 
	\end{align*}
	where \annotate.

\subsection{Proof of \Cref{thm:nonsmooth_strongly_convex_upper_bounds}} \label{app:proof_nonsmooth_strongly_convex_upper_bounds}

    Consider the problem~\eqref{eq:mixed_constraints_reformulation} in the form \eqref{eq:mixed_constraints_nonsmooth_penalized}, where $F$ satisfies Assumption~\ref{ass:nonsmooth_and_str_convex_of_f_i} with $\mu_f > 0$, $\beta, \gamma$ are defined as in Section~\ref{sec:optimal_alg_Smooth_Str_Convex} and $r$, $\alpha$, $\varepsilon$ satisfy following conditions: 
    \begin{equation}
    \label{eq:strcv_non_smooth_coupled_constraints_constants_condition_thm}
        r \leq \frac{M}{\sminp(\bB)}, \quad \alpha^2 = \frac{\mu_{\bA} + L_{\bA}}{\mu_{\bW}}, \quad \varepsilon \leq \frac{4 r^2 \mu_\bA}{\mu_f}.
    \end{equation}

As shown in Lemma~\ref{lem:strong_conv_G_penal}, $G(\bx, \by, \tilde{\bx})$ is $\frac{\mu_f}{2}$-strongly convex on $\mathcal{X} \times \mathcal{L}_m^\perp \times \widetilde{\mathcal{X}}$.
Applying the \textsc{Gradient Sliding} with restarting procedure to the problem~\eqref{eq:mixed_constraints_nonsmooth_penalized} and using the \Cref{thm:upper_bound_nonsmooth_affine_constraints}, we obtain that an \(\varepsilon\)-solution can be found using $N_K = O\left(\frac{M_f}{\sqrt{\mu_f \varepsilon}} \sqrt{\kappa_{\bK}}\right)$
multiplications by $\bK$ and $\bK^\top$, and
$
N_{\nabla F} = O\left(\frac{M_f^2}{\mu_f \varepsilon} + N_K\right)$
calls to a subgradient oracle of $F$.  Chebyshev preconditioning for $\bK$ requires $O(\sqrt{\kappaproj})$ multiplications by $\bA$, $\bA^\top$, $O(\sqrt{\kappaproj} \sqrt{\kappa_C} )$ multiplications by $\bC$, $\bC^\top$, $O\left(\sqrt{\hat{\kappa}_{\widetilde{C}^\top}}\right)$ multiplications by $\widetilde{\bC}, \widetilde{\bC}^\top$ and $O\left((\sqrt{\kappaproj} + \sqrt{\hat{\kappa}_{\widetilde{C}^\top}})\sqrt{\kappa_W} \right)$ communication rounds. Then we receive a matrix $\bK^\prime$, for which $\kappa_{\bK^\prime}=O(1)$. Hence $N_{K^\prime} = O\left(\frac{M}{\sqrt{\mu_f \varepsilon}} \right)$. From this we derived the results in the Table~\ref{table:nonsmooth_strconv_optimization_constraints}.

\end{document}

%% file: header.tex
\usepackage[english]{babel}
\usepackage{hyperref}       %
\usepackage{url}            %
\usepackage{booktabs}       %
\usepackage{amsfonts}       %
\usepackage{tablefootnote}
\usepackage{verbatim}
\usepackage{nicefrac}       %
\usepackage{microtype}      %
\usepackage{lipsum}
\usepackage{mathtools}
\usepackage{amsmath,amssymb,amsthm}
\usepackage{algorithm,algorithmic}
\usepackage{pifont}
\usepackage{cases}
\usepackage{subcaption,graphicx}
\usepackage{stackengine}    %
\usepackage{wrapfig}
\usepackage{enumitem}
\usepackage{multirow}
\usepackage{xcolor}
\usepackage{array}
\usepackage{cleveref}
\usepackage[textsize=tiny]{todonotes}

\newcommand{\tkappa}{\widetilde{\kappa}}
\newcommand{\hkappa}{\hat{\kappa}}

\newcommand{\Nnabla}{N_{\nabla f}}
\newcommand{\Na}{N_A}
\newcommand{\Nc}{N_C}
\newcommand{\Nw}{N_W}

\usepackage{aliascnt}
\newaliascnt{problem}{equation}
\makeatletter
\def\incr@probnum{\refstepcounter{problem}\let\incr@probnum\@empty}
\newenvironment{problem}{
	\incr@probnum
	\mathdisplay@push
	\st@rredfalse \global\@eqnswtrue
	\mathdisplay{equation}
}{
	\endmathdisplay{equation}
	\mathdisplay@pop
	\ignorespacesafterend
}
\makeatother

\newcounter{annotatecount}
\newcounter{annotateidx}
\newcounter{annotatejdx}
\newcounter{annotatelabelcount}

\newcommand{\atran}[2]{\stepcounter{annotatecount}\overset{(\alph{annotatecount})}{#1}\csgdef{annotatedescription\theannotatecount}{#2}}

\newcommand{\aeq}[1]{\atran{=}{#1}}
\newcommand{\aleq}[1]{\atran{\leq}{#1}}
\newcommand{\ageq}[1]{\atran{\geq}{#1}}

\newcommand{\annotateinitused}{\setcounter{annotateidx}{0}\whileboolexpr{test{\ifnumless{\theannotateidx}{\theannotatecount}}}{\stepcounter{annotateidx}\csgdef{aused\theannotateidx}{0}}}
\newcommand{\annotategetlabels}{\setcounter{annotatejdx}{0}\setcounter{annotatelabelcount}{0}\whileboolexpr{test{\ifnumless{\theannotatejdx}{\theannotatecount}}}{\stepcounter{annotatejdx}\ifcsequal{annotatedescription\theannotateidx}{annotatedescription\theannotatejdx}{\csgdef{aused\theannotatejdx}{1}\stepcounter{annotatelabelcount}\csedef{annotatelabel\theannotatelabelcount}{(\alph{annotatejdx})}}{}}}
\newcommand{\annotateprintlabels}{\setcounter{annotatejdx}{0}\whileboolexpr{test{\ifnumless{\theannotatejdx}{\theannotatelabelcount}}}{\stepcounter{annotatejdx}\ifnumequal{\theannotatejdx}{\theannotatelabelcount}{\ifnumequal{\theannotatejdx}{1}{}{~and~}}{}\csuse{annotatelabel\theannotatejdx}\ifnumless{\theannotatejdx}{\theannotatelabelcount}{\ifnumless{\theannotatejdx+1}{\theannotatelabelcount}{,~}{}}{}}}
\newcommand{\annotate}{\annotateinitused\setcounter{annotateidx}{0}\whileboolexpr{test{\ifnumless{\theannotateidx}{\theannotatecount}}}{\stepcounter{annotateidx}\ifcsstring{aused\theannotateidx}{0}{\ifnumequal{\theannotateidx}{1}{}{;~}\annotategetlabels\annotateprintlabels~\csuse{annotatedescription\theannotateidx}}{}}\setcounter{annotatecount}{0}}

\newcommand{\smax}{\sigma_{\max}}

\newcommand{\sminp}{\sigma_{\min^+}}

\newcommand{\lmax}{\lambda_{\max}}

\newcommand{\lminp}{\lambda_{\min^+}}

\newcommand{\PL}{\bP_{\cL_m}}
\newcommand{\PLb}{\bP_{\cL^\bot_m}}
\newcommand{\PC}{\bP_{\ker \bC}~}
\newcommand{\PCb}{\bP_{\ker^\bot \bC}~}

\newcommand{\eps}{\varepsilon}

\newcommand{\one}{\mathbf{1}}

\renewcommand{\hat}{\widehat}

\newcommand{\pmat}[1]{\begin{pmatrix}#1\end{pmatrix}}

\newcommand{\sumin}{\sum_{i=1}^n}

\DeclareMathOperator*{\argmin}{arg\,min}

\DeclareMathOperator*{\Argmin}{Arg\,min}

\DeclareMathOperator{\spn}{span}
\DeclareMathOperator{\image}{Im}

\DeclareMathOperator{\col}{col}

\DeclareMathOperator{\diag}{diag}

\newcommand{\N}{\mathbb{N}}
\newcommand{\R}{\mathbb{R}}

\newcommand{\bA}{{\bf A}}
\newcommand{\bB}{{\bf B}}

\newcommand{\bC}{{\bf C}}
\newcommand{\tC}{{\bf \tilde C}}
\newcommand{\bD}{{\bf D}}
\newcommand{\bE}{{\bf E}}

\newcommand{\bI}{{\bf I}}
\newcommand{\bJ}{{\bf J}}
\newcommand{\bK}{{\bf K}}
\newcommand{\bL}{{\bf L}}
\newcommand{\bM}{{\bf M}}

\newcommand{\bP}{{\bf P}}

\newcommand{\bS}{{\bf S}}

\newcommand{\bW}{{\bf W}}

\newcommand{\cE}{{\mathcal{E}}}

\newcommand{\cG}{{\mathcal{G}}}

\newcommand{\cL}{{\mathcal{L}}}
\newcommand{\cM}{{\mathcal{M}}}

\newcommand{\cU}{{\mathcal{U}}}
\newcommand{\cV}{{\mathcal{V}}}

\newcommand{\cY}{{\mathcal{Y}}}

\newcommand{\bb}{{\bf b}}
\newcommand{\bc}{{\bf c}}

\newcommand{\bv}{{\bf v}}

\newcommand{\bx}{{\bf x}}
\newcommand{\by}{{\bf y}}
\newcommand{\bz}{{\bf z}}

\newcommand{\norm}[1]{\left\| #1 \right\|}

\newcommand{\sqn}[1]{\norm{#1}_2^2}
\newcommand{\normsq}[1]{\norm{#1}^2}

\newcommand{\angles}[1]{\left\langle #1 \right\rangle}
\newcommand{\cbr}[1]{\left( #1 \right)}

\newcommand{\br}[1]{\left\{ #1 \right\}}
\def\<#1,#2>{\langle #1,#2\rangle}
\newcommand{\cbraces}[1]{\left( #1 \right)}
\newcommand{\sbraces}[1]{\left[ #1 \right]}
\newcommand{\braces}[1]{\left\{ #1 \right\}}

\newcommand{\eqdef}{\coloneqq}
\newcommand{\muproj}{\widetilde\mu_{AC}}
\newcommand{\kappaproj}{\tkappa_{AC}}

%% file: main.bbl
\begin{thebibliography}{57}
\providecommand{\natexlab}[1]{#1}
\providecommand{\url}[1]{\texttt{#1}}
\expandafter\ifx\csname urlstyle\endcsname\relax
  \providecommand{\doi}[1]{doi: #1}\else
  \providecommand{\doi}{doi: \begingroup \urlstyle{rm}\Url}\fi

\bibitem[Auzinger \& Melenk(2011)Auzinger and Melenk]{auzinger2011iterative}
Auzinger, W. and Melenk, J.
\newblock Iterative solution of large linear systems.
\newblock \emph{Lecture notes, TU Wien}, 2011.

\bibitem[Boyd et~al.(2011)Boyd, Parikh, Chu, Peleato, Eckstein,
  et~al.]{boyd2011distributed}
Boyd, S., Parikh, N., Chu, E., Peleato, B., Eckstein, J., et~al.
\newblock Distributed optimization and statistical learning via the alternating
  direction method of multipliers.
\newblock \emph{Foundations and Trends{\textregistered} in Machine learning},
  3\penalty0 (1):\penalty0 1--122, 2011.

\bibitem[Chang(2016)]{chang2016proximal}
Chang, T.-H.
\newblock A proximal dual consensus admm method for multi-agent constrained
  optimization.
\newblock \emph{IEEE Transactions on Signal Processing}, 64\penalty0
  (14):\penalty0 3719--3734, 2016.

\bibitem[Chen et~al.(2020)Chen, Jin, Sun, and Yin]{chen2020vafl}
Chen, T., Jin, X., Sun, Y., and Yin, W.
\newblock Vafl: a method of vertical asynchronous federated learning.
\newblock \emph{arXiv preprint arXiv:2007.06081}, 2020.

\bibitem[Doan \& Olshevsky(2017)Doan and Olshevsky]{doan2017distributed}
Doan, T.~T. and Olshevsky, A.
\newblock Distributed resource allocation on dynamic networks in quadratic
  time.
\newblock \emph{Systems \& Control Letters}, 99:\penalty0 57--63, 2017.

\bibitem[Du \& Meng(2023)Du and Meng]{du2023linear}
Du, K. and Meng, M.
\newblock Linear convergence of distributed aggregative optimization with
  coupled inequality constraints.
\newblock \emph{arXiv preprint arXiv:2306.06700}, 2023.

\bibitem[Du \& Meng(2025)Du and Meng]{du2025distributed}
Du, K. and Meng, M.
\newblock Distributed aggregative optimization with affine coupling
  constraints.
\newblock \emph{Neural Networks}, 184:\penalty0 107085, 2025.

\bibitem[Dvinskikh \& Gasnikov(2021{\natexlab{a}})Dvinskikh and
  Gasnikov]{dvinskikh2019decentralized}
Dvinskikh, D. and Gasnikov, A.
\newblock Decentralized and parallel primal and dual accelerated methods for
  stochastic convex programming problems.
\newblock \emph{Journal of Inverse and Ill-posed Problems}, 29\penalty0
  (3):\penalty0 385--405, 2021{\natexlab{a}}.

\bibitem[Dvinskikh \& Gasnikov(2021{\natexlab{b}})Dvinskikh and
  Gasnikov]{dvinskikh2021decentralized}
Dvinskikh, D. and Gasnikov, A.
\newblock Decentralized and parallel primal and dual accelerated methods for
  stochastic convex programming problems.
\newblock \emph{Journal of Inverse and Ill-posed Problems}, 29\penalty0
  (3):\penalty0 385--405, 2021{\natexlab{b}}.

\bibitem[Falsone et~al.(2020)Falsone, Notarnicola, Notarstefano, and
  Prandini]{falsone2020tracking}
Falsone, A., Notarnicola, I., Notarstefano, G., and Prandini, M.
\newblock Tracking-admm for distributed constraint-coupled optimization.
\newblock \emph{Automatica}, 117:\penalty0 108962, 2020.

\bibitem[Gong \& Zhang(2023)Gong and Zhang]{gong2023decentralized}
Gong, K. and Zhang, L.
\newblock Decentralized proximal method of multipliers for convex optimization
  with coupled constraints.
\newblock \emph{arXiv preprint arXiv:2310.15596}, 2023.

\bibitem[Gorbunov et~al.(2019)Gorbunov, Dvinskikh, and
  Gasnikov]{gorbunov2019optimal}
Gorbunov, E., Dvinskikh, D., and Gasnikov, A.
\newblock Optimal decentralized distributed algorithms for stochastic convex
  optimization.
\newblock \emph{arXiv preprint arXiv:1911.07363}, 2019.

\bibitem[Gutknecht \& R{\"o}llin(2002)Gutknecht and
  R{\"o}llin]{gutknecht2002chebyshev}
Gutknecht, M.~H. and R{\"o}llin, S.
\newblock The chebyshev iteration revisited.
\newblock \emph{Parallel Computing}, 28\penalty0 (2):\penalty0 263--283, 2002.

\bibitem[Hanzely \& Richt{\'a}rik(2020)Hanzely and
  Richt{\'a}rik]{hanzely2020federated}
Hanzely, F. and Richt{\'a}rik, P.
\newblock Federated learning of a mixture of global and local models.
\newblock \emph{arXiv preprint arXiv:2002.05516}, 2020.

\bibitem[Hanzely et~al.(2020)Hanzely, Hanzely, Horv{\'a}th, and
  Richt{\'a}rik]{hanzely2020lower}
Hanzely, F., Hanzely, S., Horv{\'a}th, S., and Richt{\'a}rik, P.
\newblock Lower bounds and optimal algorithms for personalized federated
  learning.
\newblock \emph{Advances in Neural Information Processing Systems},
  33:\penalty0 2304--2315, 2020.

\bibitem[Kairouz et~al.(2019)Kairouz, McMahan, Avent, Bellet, Bennis, Bhagoji,
  Bonawitz, Charles, Cormode, Cummings, et~al.]{kairouz2019advances}
Kairouz, P., McMahan, H.~B., Avent, B., Bellet, A., Bennis, M., Bhagoji, A.~N.,
  Bonawitz, K., Charles, Z., Cormode, G., Cummings, R., et~al.
\newblock Advances and open problems in federated learning.
\newblock \emph{arXiv preprint arXiv:1912.04977}, 2019.

\bibitem[Kairouz et~al.(2021)Kairouz, McMahan, Avent, Bellet, Bennis, Bhagoji,
  Bonawitz, Charles, Cormode, Cummings, et~al.]{kairouz2021advances}
Kairouz, P., McMahan, H.~B., Avent, B., Bellet, A., Bennis, M., Bhagoji, A.~N.,
  Bonawitz, K., Charles, Z., Cormode, G., Cummings, R., et~al.
\newblock Advances and open problems in federated learning.
\newblock \emph{Foundations and trends{\textregistered} in machine learning},
  14\penalty0 (1--2):\penalty0 1--210, 2021.

\bibitem[Koloskova et~al.(2020)Koloskova, Loizou, Boreiri, Jaggi, and
  Stich]{koloskova2020unified}
Koloskova, A., Loizou, N., Boreiri, S., Jaggi, M., and Stich, S.~U.
\newblock A unified theory of decentralized sgd with changing topology and
  local updates.
\newblock \emph{ICML 2020, arXiv preprint arXiv:2003.10422}, 2020.

\bibitem[Koloskova et~al.(2021)Koloskova, Lin, and
  Stich]{koloskova2021improved}
Koloskova, A., Lin, T., and Stich, S.~U.
\newblock An improved analysis of gradient tracking for decentralized machine
  learning.
\newblock \emph{Advances in Neural Information Processing Systems}, 34, 2021.

\bibitem[Kovalev et~al.(2020)Kovalev, Salim, and
  Richt{\'a}rik]{kovalev2020optimal}
Kovalev, D., Salim, A., and Richt{\'a}rik, P.
\newblock Optimal and practical algorithms for smooth and strongly convex
  decentralized optimization.
\newblock \emph{Advances in Neural Information Processing Systems}, 33, 2020.

\bibitem[Kovalev et~al.(2021{\natexlab{a}})Kovalev, Gasanov, Gasnikov, and
  Richtarik]{kovalev2021lower}
Kovalev, D., Gasanov, E., Gasnikov, A., and Richtarik, P.
\newblock Lower bounds and optimal algorithms for smooth and strongly convex
  decentralized optimization over time-varying networks.
\newblock \emph{Advances in Neural Information Processing Systems}, 34,
  2021{\natexlab{a}}.

\bibitem[Kovalev et~al.(2021{\natexlab{b}})Kovalev, Shulgin, Richt{\'a}rik,
  Rogozin, and Gasnikov]{kovalev2021adom}
Kovalev, D., Shulgin, E., Richt{\'a}rik, P., Rogozin, A.~V., and Gasnikov, A.
\newblock Adom: Accelerated decentralized optimization method for time-varying
  networks.
\newblock In \emph{International Conference on Machine Learning}, pp.\
  5784--5793. PMLR, 2021{\natexlab{b}}.

\bibitem[Kovalev et~al.(2024)Kovalev, Borodich, Gasnikov, and
  Feoktistov]{kovalev2024lower}
Kovalev, D., Borodich, E., Gasnikov, A., and Feoktistov, D.
\newblock Lower bounds and optimal algorithms for non-smooth convex
  decentralized optimization over time-varying networks.
\newblock \emph{Advances in Neural Information Processing Systems},
  37:\penalty0 96566--96606, 2024.

\bibitem[Lan(2019)]{Lan2019lectures}
Lan, G.
\newblock Lectures on optimization methods for machine learning.
\newblock \emph{e-print}, 2019.

\bibitem[Lan(2020)]{lan2020first}
Lan, G.
\newblock \emph{First-order and Stochastic Optimization Methods for Machine
  Learning}.
\newblock Springer, 2020.

\bibitem[Lan et~al.(2020)Lan, Lee, and Zhou]{lan2020communication}
Lan, G., Lee, S., and Zhou, Y.
\newblock Communication-efficient algorithms for decentralized and stochastic
  optimization.
\newblock \emph{Mathematical Programming}, 180\penalty0 (1):\penalty0 237--284,
  2020.

\bibitem[Li \& Lin(2021)Li and Lin]{li2021accelerated}
Li, H. and Lin, Z.
\newblock Accelerated gradient tracking over time-varying graphs for
  decentralized optimization.
\newblock \emph{arXiv preprint arXiv:2104.02596}, 2021.

\bibitem[Li et~al.(2018)Li, L{\"u}, Liao, and Huang]{li2018accelerated}
Li, H., L{\"u}, Q., Liao, X., and Huang, T.
\newblock Accelerated convergence algorithm for distributed constrained
  optimization under time-varying general directed graphs.
\newblock \emph{IEEE Transactions on Systems, Man, and Cybernetics: Systems},
  50\penalty0 (7):\penalty0 2612--2622, 2018.

\bibitem[Lian et~al.(2017)Lian, Zhang, Zhang, Hsieh, Zhang, and
  Liu]{lian2017can}
Lian, X., Zhang, C., Zhang, H., Hsieh, C.-J., Zhang, W., and Liu, J.
\newblock Can decentralized algorithms outperform centralized algorithms? a
  case study for decentralized parallel stochastic gradient descent.
\newblock In \emph{Advances in Neural Information Processing Systems}, pp.\
  5330--5340, 2017.

\bibitem[Liang et~al.(2019)Liang, Yin, et~al.]{liang2019distributed}
Liang, S., Yin, G., et~al.
\newblock Distributed smooth convex optimization with coupled constraints.
\newblock \emph{IEEE Transactions on Automatic Control}, 65\penalty0
  (1):\penalty0 347--353, 2019.

\bibitem[Liu et~al.(2024)Liu, Kang, Zou, Pu, He, Ye, Ouyang, Zhang, and
  Yang]{liu2024vertical}
Liu, Y., Kang, Y., Zou, T., Pu, Y., He, Y., Ye, X., Ouyang, Y., Zhang, Y.-Q.,
  and Yang, Q.
\newblock Vertical federated learning: Concepts, advances, and challenges.
\newblock \emph{IEEE transactions on knowledge and data engineering},
  36\penalty0 (7):\penalty0 3615--3634, 2024.

\bibitem[Makhija et~al.(2022)Makhija, Ho, and Ghosh]{makhija2022federated}
Makhija, D., Ho, N., and Ghosh, J.
\newblock Federated self-supervised learning for heterogeneous clients.
\newblock \emph{arXiv preprint arXiv:2205.12493}, 2022.

\bibitem[McMahan et~al.(2017)McMahan, Moore, Ramage, Hampson, and
  y~Arcas]{mcmahan2017communication}
McMahan, B., Moore, E., Ramage, D., Hampson, S., and y~Arcas, B.~A.
\newblock Communication-efficient learning of deep networks from decentralized
  data.
\newblock In \emph{Artificial intelligence and statistics}, pp.\  1273--1282.
  PMLR, 2017.

\bibitem[Necoara \& Nedelcu(2014)Necoara and Nedelcu]{necoara2014distributed}
Necoara, I. and Nedelcu, V.
\newblock Distributed dual gradient methods and error bound conditions.
\newblock \emph{arXiv preprint arXiv:1401.4398}, 2014.

\bibitem[Necoara \& Nedelcu(2015)Necoara and Nedelcu]{necoara2015linear}
Necoara, I. and Nedelcu, V.
\newblock On linear convergence of a distributed dual gradient algorithm for
  linearly constrained separable convex problems.
\newblock \emph{Automatica}, 55:\penalty0 209--216, 2015.

\bibitem[Necoara et~al.(2011)Necoara, Nedelcu, and
  Dumitrache]{necoara2011parallel}
Necoara, I., Nedelcu, V., and Dumitrache, I.
\newblock Parallel and distributed optimization methods for estimation and
  control in networks.
\newblock \emph{Journal of Process Control}, 21\penalty0 (5):\penalty0
  756--766, 2011.

\bibitem[Nedi{\'c} \& Ozdaglar(2009)Nedi{\'c} and
  Ozdaglar]{nedic2009distributed}
Nedi{\'c}, A. and Ozdaglar, A.
\newblock Distributed subgradient methods for multi-agent optimization.
\newblock \emph{IEEE Transactions on Automatic Control}, 54\penalty0
  (1):\penalty0 48--61, 2009.

\bibitem[Nedic et~al.(2010)Nedic, Ozdaglar, and Parrilo]{nedic2010constrained}
Nedic, A., Ozdaglar, A., and Parrilo, P.~A.
\newblock Constrained consensus and optimization in multi-agent networks.
\newblock \emph{IEEE Transactions on Automatic Control}, 55\penalty0
  (4):\penalty0 922--938, 2010.

\bibitem[Nedic et~al.(2017)Nedic, Olshevsky, and Shi]{nedic2017achieving}
Nedic, A., Olshevsky, A., and Shi, W.
\newblock Achieving geometric convergence for distributed optimization over
  time-varying graphs.
\newblock \emph{SIAM Journal on Optimization}, 27\penalty0 (4):\penalty0
  2597--2633, 2017.

\bibitem[Nedi{\'c} et~al.(2018)Nedi{\'c}, Olshevsky, and
  Shi]{nedic2018improved}
Nedi{\'c}, A., Olshevsky, A., and Shi, W.
\newblock Improved convergence rates for distributed resource allocation.
\newblock In \emph{2018 IEEE Conference on Decision and Control (CDC)}, pp.\
  172--177. IEEE, 2018.

\bibitem[Rogozin et~al.(2022)Rogozin, Yarmoshik, Kopylova, and
  Gasnikov]{rogozin2022decentralized_affine}
Rogozin, A., Yarmoshik, D., Kopylova, K., and Gasnikov, A.
\newblock Decentralized strongly-convex optimization with affine constraints:
  Primal and dual approaches.
\newblock In \emph{International Conference on Optimization and Applications},
  pp.\  93--105. Springer, 2022.

\bibitem[Salim et~al.(2022)Salim, Condat, Kovalev, and
  Richt{\'a}rik]{salim2022optimal}
Salim, A., Condat, L., Kovalev, D., and Richt{\'a}rik, P.
\newblock An optimal algorithm for strongly convex minimization under affine
  constraints.
\newblock In \emph{International conference on artificial intelligence and
  statistics}, pp.\  4482--4498. PMLR, 2022.

\bibitem[Scaman et~al.(2017)Scaman, Bach, Bubeck, Lee, and
  Massouli{\'e}]{scaman2017optimal}
Scaman, K., Bach, F., Bubeck, S., Lee, Y.~T., and Massouli{\'e}, L.
\newblock Optimal algorithms for smooth and strongly convex distributed
  optimization in networks.
\newblock In \emph{Proceedings of the 34th International Conference on Machine
  Learning-Volume 70}, pp.\  3027--3036. JMLR. org, 2017.

\bibitem[Scaman et~al.(2018)Scaman, Bach, Bubeck, Massouli{\'e}, and
  Lee]{scaman2018optimal}
Scaman, K., Bach, F., Bubeck, S., Massouli{\'e}, L., and Lee, Y.~T.
\newblock Optimal algorithms for non-smooth distributed optimization in
  networks.
\newblock In \emph{Advances in Neural Information Processing Systems}, pp.\
  2740--2749, 2018.

\bibitem[Stanko et~al.(2026)Stanko, Karimullin, Beznosikov, and
  Gasnikov]{stanko2026accelerated}
Stanko, S., Karimullin, T., Beznosikov, A., and Gasnikov, A.
\newblock Accelerated methods with compression for horizontal and vertical
  federated learning: S. stanko et al.
\newblock \emph{Journal of Optimization Theory and Applications}, 208\penalty0
  (2):\penalty0 68, 2026.

\bibitem[Uribe et~al.(2020)Uribe, Lee, Gasnikov, and Nedi{\'c}]{uribe2020dual}
Uribe, C.~A., Lee, S., Gasnikov, A., and Nedi{\'c}, A.
\newblock A dual approach for optimal algorithms in distributed optimization
  over networks.
\newblock \emph{Optimization Methods and Software}, pp.\  1--40, 2020.

\bibitem[Vepakomma et~al.(2018)Vepakomma, Gupta, Swedish, and
  Raskar]{vepakomma2018split}
Vepakomma, P., Gupta, O., Swedish, T., and Raskar, R.
\newblock Split learning for health: Distributed deep learning without sharing
  raw patient data.
\newblock \emph{arXiv preprint arXiv:1812.00564}, 2018.

\bibitem[Wang \& Hu(2022)Wang and Hu]{wang2022distributed}
Wang, J. and Hu, G.
\newblock Distributed optimization with coupling constraints in multi-cluster
  networks based on dual proximal gradient method.
\newblock \emph{arXiv preprint arXiv:2203.00956}, 2022.

\bibitem[Wang et~al.(2016)Wang, Kolar, and Srerbo]{wang2016distributed}
Wang, J., Kolar, M., and Srerbo, N.
\newblock Distributed multi-task learning.
\newblock In \emph{Artificial intelligence and statistics}, pp.\  751--760.
  PMLR, 2016.

\bibitem[Wu et~al.(2022)Wu, Wang, and Lu]{wu2022distributed}
Wu, X., Wang, H., and Lu, J.
\newblock Distributed optimization with coupling constraints.
\newblock \emph{IEEE Transactions on Automatic Control}, 68\penalty0
  (3):\penalty0 1847--1854, 2022.

\bibitem[Xie et~al.(2024)Xie, Chen, Li, Nourian, Zhang, and
  Li]{xie2024improving}
Xie, C., Chen, P.-Y., Li, Q., Nourian, A., Zhang, C., and Li, B.
\newblock Improving privacy-preserving vertical federated learning by efficient
  communication with admm.
\newblock In \emph{2024 IEEE Conference on Secure and Trustworthy Machine
  Learning (SaTML)}, pp.\  443--471. IEEE, 2024.

\bibitem[Yang et~al.(2019)Yang, Liu, Chen, and Tong]{yang2019federated}
Yang, Q., Liu, Y., Chen, T., and Tong, Y.
\newblock Federated machine learning: Concept and applications.
\newblock \emph{ACM Transactions on Intelligent Systems and Technology (TIST)},
  10\penalty0 (2):\penalty0 1--19, 2019.

\bibitem[Yarmoshik et~al.(2024{\natexlab{a}})Yarmoshik, Rogozin, and
  Gasnikov]{yarmoshik2024decentralized_local}
Yarmoshik, D., Rogozin, A., and Gasnikov, A.
\newblock Decentralized optimization with affine constraints over time-varying
  networks.
\newblock \emph{Computational Management Science}, 21\penalty0 (1):\penalty0
  10, 2024{\natexlab{a}}.

\bibitem[Yarmoshik et~al.(2024{\natexlab{b}})Yarmoshik, Rogozin, Kiselev,
  Dorin, Gasnikov, and Kovalev]{yarmoshik2024decentralized}
Yarmoshik, D., Rogozin, A., Kiselev, N., Dorin, D., Gasnikov, A., and Kovalev,
  D.
\newblock Decentralized optimization with coupled constraints.
\newblock \emph{arXiv preprint arXiv:2407.02020}, 2024{\natexlab{b}}.

\bibitem[Zhang et~al.(2021)Zhang, Gu, and Li]{zhang2021distributed}
Zhang, B., Gu, C., and Li, J.
\newblock Distributed convex optimization with coupling constraints over
  time-varying directed graphs.
\newblock \emph{Journal of Industrial and Management Optimization}, 17\penalty0
  (4):\penalty0 2119--2138, 2021.

\bibitem[Zhang et~al.(2015)Zhang, Choromanska, and LeCun]{zhang2015deep}
Zhang, S., Choromanska, A.~E., and LeCun, Y.
\newblock Deep learning with elastic averaging sgd.
\newblock \emph{Advances in neural information processing systems}, 28, 2015.

\bibitem[Zhu \& Martinez(2011)Zhu and Martinez]{zhu2011distributed}
Zhu, M. and Martinez, S.
\newblock On distributed convex optimization under inequality and equality
  constraints.
\newblock \emph{IEEE Transactions on Automatic Control}, 57\penalty0
  (1):\penalty0 151--164, 2011.

\end{thebibliography}
